\documentclass[12pt,reqno]{amsart}

\newcommand\version{October 20, 2026}

\usepackage{amsmath,amsfonts,amsthm,amssymb,amsxtra}
\usepackage{hyperref}
\usepackage{bbm} % to get \1
\usepackage{bm} % to get \bm
\usepackage{accents} % to get double tilde
\newtheorem{theorem}{Theorem}[section]
\newtheorem{proposition}[theorem]{Proposition}
\newtheorem{lemma}[theorem]{Lemma}

\newtheorem*{problem*}{Problem}

\theoremstyle{definition}

\theoremstyle{remark}

\newtheorem{remark}[theorem]{Remark}

\numberwithin{equation}{section}

\newcommand{\B}{\mathcal{B}}

\newcommand{\C}{\mathbb{C}}

\newcommand{\D}{\mathcal{D}}

\renewcommand{\epsilon}{\varepsilon}

\newcommand{\N}{\mathbb{N}}

\renewcommand{\phi}{\varphi}
\newcommand{\R}{\mathbb{R}}

\DeclareMathOperator{\ad}{ad}

\DeclareMathOperator{\im}{Im}

\DeclareMathOperator{\ran}{ran}
\DeclareMathOperator{\re}{Re}
\DeclareMathOperator*{\slim}{s-lim}

\DeclareMathOperator{\supp}{supp}

\def\ce{\mathcal{E}}

\newcommand{\me}[1]{\mathrm{e}^{#1}}
\newcommand{\one}{\mathbf{1}}

\begin{document}

\title[Maximal velocity estimates for Hardy operators --- \version]{Maximal velocity estimates for Hardy operators}

\author[A. Ishida]{Atsuhide Ishida}
\address[Atsuhide Ishida]{Katsushika Division, Institute of Arts and Sciences, Tokyo University of Science, 6-3-1 Niijuku, Katsushika-ku, Tokyo 125-8585, Japan}
\email{aishida@rs.tus.ac.jp}

\author[K. Merz]{Konstantin Merz}
\address[Konstantin Merz]{Institute for Theoretical Physics, ETH Zurich, Wolfgang-Pauli-Strasse 27, 8093 Zurich, Switzerland, and Institut f\"ur Analysis und Algebra, Technische Universit\"at Braunschweig, Universit\"atsplatz 2, 38106 Braun\-schweig, Germany}
\email{k.merz@tu-bs.de, k.merz@phys.ethz.ch}

% 35Q40 PDEs in connection with quantum mechanics
% 46N50 Applications of functional analysis in quantum physics
% 47A40 Scattering theory of linear operators 
% 35P25 Scattering theory for PDEs 
% 81U05 2-body potential quantum scattering theory


\date{\version}

\begin{abstract}
  We prove maximal velocity estimates for the fractional Laplacian with an added subcritical Hardy potential, called Hardy operator, and possibly external potentials. As an application, we give another proof of the nonexistence of the wave operator intertwining Hardy operators with and without long-range perturbations with Coulomb decay.
\end{abstract}

\maketitle
% \setcounter{tocdepth}{1}
% \tableofcontents

%%%%%%%%%%%%%%%%%%%%%%%%%%%%%%%%%%%%%%%%
%%%%%%%%%%%%%%%%%%%%%%%%%%%%%%%%%%%%%%%%

\section{Introduction and main results}

Let $d\in\N=\{1,2,...\}$, $\alpha\in(0,d)\cap(0,2]$, and $D=-i\nabla$. Let
\begin{align}
  \kappa_{\rm c} = \kappa_{\rm c}(d,\alpha) = \frac{2^\alpha\Gamma\left(\frac{d+\alpha}{4}\right)^2}{\Gamma\left(\frac{d-\alpha}{4}\right)^2}
\end{align}
be the sharp constant in the Hardy--Kato--Herbst inequality \cite{Kato1966,Herbst1977}, for short Hardy inequality, which states
\begin{align}
  |D|^\alpha \geq \kappa_{\rm c}|x|^{-\alpha}
\end{align}
in the sense of quadratic forms on $L^2(\R^d)$; see also \cite{Kovalenkoetal1981,Yafaev1999,Franketal2008H,FrankSeiringer2008,Bogdanetal2016,BogdanMerz2024} for different proofs. For $\kappa\leq\kappa_{\rm c}$, we consider the Friedrichs extension
\begin{align}
  L_\kappa = |D|^{\alpha}-\frac{\kappa}{|x|^\alpha} \quad \text{in} \ L^2(\R^d)
\end{align}
of the corresponding expression on $C_c^\infty(\R^d)$. When $\kappa<\kappa_{\rm c}$, then its form domain is the inhomogeneous Sobolev space $H^{\alpha/2}(\R^d)$.
We call $\kappa/|x|^\alpha$ the Hardy potential and $L_\kappa$ the Hardy operator. If $\alpha=1$, $L_\kappa$ is an important model in relativistic quantum mechanics, where it describes, e.g., electrons of a heavy atom in the vicinity of the nucleus, see, e.g., \cite{Franketal2008,Franketal2020P}, or in the study of gravitational collapse of stars, see, e.g., \cite{Chandrasekhar1931,LiebThirring1984,LiebYau1987,LiebYau1988,FrohlichLenzmann2007D}.

In \cite{IshidaMerz2026}, we study the time-dependent wave operators intertwining $L_0$ and $L_\kappa$ with possibly added external perturbations. In particular, we show non-existence of wave operators when the external potentials of the ``perturbed'' Hamiltonian involve long-range perturbations compared to the ``reference'' or ``free'' Hamiltonian. Our first attempts used a (time-integrated) {\em maximal velocity estimate}. As we eventually found a more streamlined approach and the proof of this maximal velocity estimate is not short, we did neither formally state, nor prove it in \cite{IshidaMerz2026}. However, as this estimate might be useful in other aspects of scattering theory for $L_\kappa$, in particular for three or more particles---see the foundational works~\cite{SigalSoffer1987,Graf1990,Derezinski1993} in the local (nonrelativistic) case $\alpha=2$---we shall state and prove it in this paper. We will also show how it can be used to prove the nonexistence of wave operators intertwining $L_\kappa$ with and without potentials decaying like the Coulomb potential. Stronger results using different propagation estimates are proved in \cite{IshidaMerz2026}.

To state our results, we introduce some notation and assumptions. We write $\langle x\rangle:=(1+|x|^2)^{1/2}$ for $x\in\R^d$ or a selfadjoint operator. Since all potentials considered here are real-valued, we just write $V\in L^\infty(\R^d)$ instead of $V\in L^\infty(\R^d;\R)$.
Some of our statements use the following smoothness assumptions
\begin{align}
  \label{ass5}
  |\partial_x^\beta V(x)| \lesssim \langle x\rangle^{-\rho-|\beta|}, \quad \rho>0, \ \beta\in\N_0^d.
\end{align}
Here and below, we write $A\lesssim B$ for functions $A,B\geq0$ to indicate that there is a constant $c$ such that $A\leq c B$. If $c$ depends on a parameter $\tau$, we may write $A\lesssim_\tau B$, but the dependence on $d,\alpha,\kappa$ is usually ignored in the notation. The notation $A\sim B$ means $A\lesssim B\lesssim A$. Furthermore, $\|A\|$ denotes the operator norm of a linear operator $A$ on a Hilbert space.
Further, we denote the characteristic function of a set $\Omega$ by
\begin{align*}
  x \mapsto
  \one_\Omega(x) :=
  \begin{cases}
    1 & \quad \text{for} \ x \in \Omega, \\
    0 & \quad \text{for} \ x \notin \Omega.
  \end{cases}
\end{align*}
For a selfadjoint operator, we use the same notation for the corresponding spectral projection.
For $x,\theta\in\R$, we sometimes write $\one(x\leq\theta)$ or $\one_{x\leq\theta}$ to mean $\one_{(-\infty,\theta]}(x)$.

%%%%%%%%%%%%%%%%%%

\subsection{Maximal velocity estimates}

We now state our first main result.

\begin{theorem}
  \label{the7}
  Let $d\in\N$, $\alpha\in[1,2]\cap(2-d/2,d)$, and $\kappa<\kappa_{\rm c}$. If $\alpha=1$, assume in addition $d\geq5$. Assume $V\in L^\infty(\R^d)$ obeys \eqref{ass5} for $\beta=0$ and some $\rho>0$. Let $H_\kappa=L_\kappa-V$.
  Then for every $f\in C_c^\infty(0,\infty)$, there is $\Theta_0=\Theta_0(d,\alpha,\kappa,\rho,f,V)>0$ such that for all $\gamma>1/2$, $\Theta \geq \Theta_0$, and $\phi \in\D(\log(2+\langle x\rangle)^\gamma)$,
  \begin{equation}
    \label{eq:the7}
    \int_1^\infty \left\|\one\left(\frac{|x|}{t}\geq\Theta\right) f(H_\kappa) \me{-i t H_\kappa}\phi\right\|^2 \, \frac{dt}{t}
    \lesssim_{\Theta_0,\gamma} \left\|\left(\log(2+\langle x\rangle)\right)^\gamma \phi \right\|^2.
  \end{equation}
  Moreover, for all $\Theta\geq2\Theta_0$,
  \begin{equation}
    \label{eq:the7cor}
    \lim_{t\rightarrow\infty}\left\|\one\left(\frac{|x|}{t}\geq\Theta\right)f(H_\kappa)\me{-itH_\kappa}\phi\right\|=0, \quad \phi\in L^2(\R^d).
  \end{equation}
\end{theorem}

In Appendix~\ref{s:sharp}, we show that \eqref{eq:the7} cannot hold uniformly in the stated class with an exponent $\gamma<1/2$. On the other hand, we do not know if $\gamma>1/2$ is optimal.

\smallskip
In \cite{Soffer2016}, maximal and minimal escape velocity bounds for $|D|-V$ in $d=3$ are proved under the assumption that $V$, $x\cdot\nabla V$, and $\langle x\rangle\Delta V$ decay faster than $\langle x\rangle^{-2}$. The maximal velocity estimate there is energy-localized down to zero and assumes that zero is neither a resonance nor an eigenvalue. Moreover, it requires finiteness of $\|\langle x\rangle^{1+\epsilon}\phi\|$ ($\epsilon>0$). On the other hand, Theorem~\ref{the7} incorporates subcritical Hardy singularities and applies to states which are energy-localized away from zero with $\|\left(\log(2+\langle x\rangle)\right)^\gamma\phi\|<\infty$ ($\gamma>1/2$). However, our result at $\alpha=1$ requires $d\geq5$.
We also mention the recent maximal velocity estimates in \cite{SigalWu2025} of the form $\|\chi_X\me{-itH}\chi_Y\|\lesssim\me{-\mu(d_{XY}-c|t|)}$ for disjoint sets $X,Y\subset\R^d$ and a broad family of Schr\"odinger-type operators $H=T(-i\nabla)-V(x)$. It excludes, however, symbols like $T=|D|^\alpha$ with $0<\alpha<2$ as they are not analytic at $\xi=0$ and, moreover, have unbounded group velocity when $\alpha\neq1$.
The estimate applies, e.g., to $(m^2-\Delta)^{1/2}-V(x)$ with $m^2>0$ and $V$ relatively form-bounded with respect to the kinetic energy with form bound less than $1$.
The estimates in \cite{SigalWu2025} are uniform off-diagonal exponential light-cone bounds, i.e., of a different nature from the estimate in Theorem~\ref{the7}.

\smallskip
Our proof of Theorem~\ref{the7} proceeds by controlling the time evolution of a ``switch function'', which turns on when $|x|/t \gtrsim \Theta$.
To this end, we compute its Heisenberg derivative, which involves a commutator with $H_\kappa$. Controlling it is not easy since $|D|^\alpha$ is nonlocal for $\alpha<2$.
To estimate the commutator, we use the Helffer--Sj\"ostrand formula, which reduces the problem to estimating commutators of $H_\kappa$ with resolvents of $|x|^2/t^2$.
A fundamental difficulty is that the commutator produces lower-order terms involving the velocity $\alpha|D|^{\alpha-2}D$. The assumption $\alpha\geq1$ ensures that the corresponding symbol is bounded at low frequencies. Note that for $\alpha<1$, even a spectral localization with respect to $H_\kappa$ does not guarantee control of $|D|^{\alpha-2}D$ at low frequencies. Treating this regime requires a separate low-frequency analysis, which we have not attempted.
In estimating the commutator terms, we are actually facing second-order commutators, which produce terms of the form $\partial_{\xi_j}\partial_{\xi_k}|\xi|^\alpha(D)\langle x\rangle^{\alpha-2}$. To control them, we use Hardy's inequality (in Fourier space), which requires $\alpha>2-d/2$. These commutator estimates yield a time-decay of order $t^{-\alpha}$, which is integrable at large times when $\alpha>1$.
To treat $\alpha=1$, we produce further time-decay by exploiting the energy-localization more efficiently. Our approach involves commutators of position operators with order-$\alpha-2$ Fourier multipliers, which thus requires estimating terms of the form $|D|^{\alpha-3}\langle x\rangle^{\alpha-3}$. Doing so is possible by using Hardy's inequality with $\alpha>3-d/2$, i.e., $d\geq5$.
These assumptions on $\alpha$ are unrelated to the Hardy potential. In fact, since we assume $\kappa<\kappa_{\rm c}$, the Hardy potential can be completely absorbed by the kinetic energy in the sense of quadratic forms.

%%%%%%%%%%%%%%%%%%

\subsection{Application: Nonexistence of a wave operator involving Coulomb-type decay}

We will use Theorem~\ref{the7} to give another proof of nonexistence of wave operators involving long-range potentials. Recall that for two selfadjoint operators $H$ and $H_0$ on a Hilbert space $\mathcal H$, the wave operators are defined as the strong limits
\begin{align}
  W_\pm(H,H_0) := \slim_{t\to\pm\infty}\me{itH}\me{-itH_0}P_{\rm ac}(H_0)
\end{align}
provided these limits exist. Here, $P_{\rm ac}(H_0)$ denotes the spectral projection on the absolutely continuous spectral subspace of $H_0$, denoted by $\mathcal H_{\rm ac}(H_0)$. As the analysis of $W_-(H,H_0)$ is under our assumptions the same as that of $W_+(H,H_0)$, we will from now on only consider the latter and denote it by $W(H,H_0)$. If $W(H,H_0)$ exists, then $\ran(W(H,H_0))$ is a closed subspace of $\mathcal H_{\rm ac}(H)$, and we have the intertwining property $\phi(H)W(H,H_0)=W(H,H_0)\phi(H_0)$ for any bounded Borel function $\phi$.

\begin{theorem}
  \label{asympcomplhardyplusperturbationnonexistence}
  Let $d\in\N$, $\alpha\in[1,2]\cap(2-d/2,d)$, and $\kappa<\kappa_{\rm c}$.
  If $\alpha=1$, assume in addition $d\geq5$.
  Assume $V^L \in C^2(\R^d)$ satisfies for some $K_1,K_2>0$,
  \begin{subequations}
    \label{eq:basicassumplongrange}
    \begin{align}
      \label{ass1}
      K_1\langle x\rangle^{-1} \leq V^L(x) \leq K_2 \langle x\rangle^{-1} \quad \text{or} \\
      \label{ass2}
      K_1\langle x\rangle^{-1} \leq -V^L(x) \leq K_2 \langle x\rangle^{-1},
    \end{align}
  \end{subequations}
  and \eqref{ass5} for $\rho=1$ and all multi-indices $\beta$ with $|\beta|\leq2$. Then $W(L_\kappa - V^L,L_\kappa)$ does not exist.
\end{theorem}

%%%%%%%%%%%%%%%%%%

\subsection*{Organization and notation}
Section~\ref{s:technicalpreliminaries} contains technical preliminaries, including the Helffer--Sj\"ostrand formula and weighted estimates.
In Sections~\ref{s:maximalvelocity} and~\ref{s:proofthe7cor}, we prove \eqref{eq:the7} and \eqref{eq:the7cor} in Theorem~\ref{the7}.
In Section~\ref{s:proofnonexistence}, we prove Theorem~\ref{asympcomplhardyplusperturbationnonexistence}. Displayed constants like $C$ in estimates may change from estimate to estimate, but we continue to use the same letter by an abuse of notation.

%%%%%%%%%%%%%%%%%%%%%%%%%%%%%%%%%%%%%%%%
%%%%%%%%%%%%%%%%%%%%%%%%%%%%%%%%%%%%%%%%

\section{Technical preliminaries}
\label{s:technicalpreliminaries}

We prepare some technical results, partly taken from \cite{IshidaMerz2026}.
We denote the generator of dilations by
\begin{align}
  A := \frac12\left(x\cdot D + D\cdot x\right).
\end{align}

%%%%%%%%%%%%%%%%%%

\subsection{Commutator classes}

The following notation will be helpful in the ensuing commutator expansions.
For a selfadjoint operator $A:\D(A)\to\mathcal H$ in a Hilbert space $\mathcal H$ and $B\in\mathcal B(\mathcal H)$ (the space of linear bounded operators on $\mathcal H$), we write $B\in C^1(A)$ if the map $\R\ni t\mapsto \me{-itA}B\me{itA}$ is strongly continuously differentiable. Equivalently, $B\mathcal D(A)\subset\mathcal D(A)$ and the commutator
$$
[B,A]=BA-AB,
\qquad
\mathcal D([B,A])=\mathcal D(A),
$$
extends to a bounded operator on $\mathcal H$.
Moreover, for $k\in\N$, we write $B\in C^k(A)$ if
$$
\R\ni t\mapsto\me{-itA}B\me{itA}
$$
is strongly $C^k$. It will also be useful to use the iterated commutator notation
\begin{align}
  {\rm ad}_{A}^{0}(H) = H
  \quad \text{and} \quad
  i^n{\rm ad}_{A}^n(H) = i[i^{n-1} {\rm ad}_{A}^{n-1}(H),A]
  \quad \text{for } n\in\N.
\end{align}

If $H$ is selfadjoint (not necessarily bounded), we write $H\in C^k(A)$ if, for one (and hence every) $z\in\rho(H)$ (the resolvent set of $H$), we have $(z-H)^{-1}\in C^k(A)$.
Whenever the commutator $[H,A]$ is represented by an operator for which the following products are meaningful, one has
$$
[(z-H)^{-1},A] = (z-H)^{-1}[H,A](z-H)^{-1}.
$$

%%%%%%%%%%%%%%%%%%

\subsection{Helffer--Sj\"ostrand formulae and commutator expansions}
\label{ss:functionalcalculus}

We recall the Helffer--Sj\"ostrand formula \cite[Proposition~7.2]{HelfferSjostrand1989}. It allows us to control functions of selfadjoint operators in terms of their resolvents. To state the formula, we recall the notation
\begin{align}
  \bar\partial
  = \frac{\partial}{\partial\overline z}
  = \frac12\left(\frac{\partial}{\partial x} + i \frac{\partial}{\partial y}\right),
  \qquad
  z=x+iy \in \C,
\end{align}
and $dz\wedge d\overline z=-2i dx\,dy$, which is $-2i$ times the Lebesgue measure on $\C$.

\begin{proposition}
  \label{prop:analyticextension}
  Let $f\in C_c^\infty(\R)$. Then there exists $\tilde f\in C_c^\infty(\C)$, called an {\em almost analytic extension of $f$} such that
  $$
  \tilde f\big|_{\R} = f
  $$
  and for every $N\in\N$ there is a constant $C_N>0$ such that
  $$
  \left| \frac{\partial \tilde f}{\partial\overline z}(z) \right|
  \leq C_N\, |\im(z)|^{N},  \quad z\in\C.
  $$
  Moreover, for every selfadjoint operator $A$ and every such $f$, we have the Helffer--Sj\"ostrand formula
  \begin{align}
    \label{eq:hs}
    f(A) = \frac{1}{2\pi i} \int_{\C}\frac{\partial\tilde f}{\partial \overline z}(z)(z-A)^{-1} \, dz\wedge d\overline z,
  \end{align}
  where the integral converges absolutely in operator norm.
  More generally, for $n\in\N$,
  \begin{align}
    \label{eq:hsderivative}
    f^{(n)}(A) = \frac{n!}{2\pi i}\int_{\C}\frac{\partial\tilde f}{\partial \overline z}(z)(z-A)^{-n-1}\, dz\,\wedge d\overline z.
  \end{align}
\end{proposition}

\begin{remark}
  One construction of an almost analytic extension of $f\in C^\infty(\R:\,\C)$ of order $N$ is
  $$
  \tilde f(z) := \left(\sum_{r=0}^N f^{(r)}(x)(iy)^r/r!\right) \eta\left(\frac{y}{\langle x\rangle}\right),
  $$
  where $\eta \in C_c^\infty(\R;[0,1])$ satisfies $\eta(s)=1$ for $|s|\leq1$, $\eta(s)=0$ for $|s|>2$, and is smooth in between; see \cite{Davies1996,DimassiSjostrand1999}. The precise values of $N$ and choice of $\eta$ are not relevant.
  It is sometimes useful to indicate the order $N$ of the above construction, in which case we write more explicitly $\widetilde f_N(z)$ instead of just $\tilde f(z)$.
  Note that $\tilde f(z)$ is smooth as a function of $x$ and $y$, but not analytic, and that $\tilde f(z)|_\R = f(x)$. It follows from this construction that
  $$
  \left|\frac{\partial \tilde f}{\partial\overline z}\right|
  = \mathcal{O}(|\im(z)|^N) \quad \text{as} \ y \to 0 \ \text{for all} \ x\in\R.
  $$
  In particular,
  $$
  \frac{\partial \tilde f}{\partial\overline z}=0, \quad z\in \R,
  $$
  which is why one calls $\tilde f$ an almost analytic extension of $f$.
\end{remark}

We will use the following two commutator estimates in the proof of nonexistence of wave operators.

\begin{lemma}[{\cite[Lemma~5.7]{IshidaMerz2026}}]
  \label{lem:commutators-spectral-cutoff}
  Let $\alpha\in[1,2]\cap(0,d)$, $\kappa<\kappa_{\rm c}$, and $\rho=1$.
  Assume $V^L \in C^2(\R^d)$ obeys \eqref{eq:basicassumplongrange} and \eqref{ass5} for all multi-indices $\beta$ with $|\beta|\leq2$. Let $H_\kappa=L_\kappa-V^L$ and $\chi\in C^\infty(\R)$ satisfy
  $$
  \chi(s)=0 \quad\text{for} \ s \leq \frac{\theta}{2},
  \qquad
  \chi(s)=1 \quad\text{for} \ s \geq \theta.
  $$
  Let $\lambda\in\R$ be such that $H_\kappa+\lambda\geq\delta>0$ for some $\delta>0$. Then,
  \begin{align}
    \label{eq:chi-resolvent-commutator}
    \|[\chi(A/t),(\lambda+H_\kappa)^{-1}]\| + \|[\chi(A/t),V^L]\| & \lesssim \frac1t, \quad t>0.
  \end{align}
\end{lemma}

We also need the boundedness of $AV^L$ for $V^L$ decaying like $|x|^{-1}$ at infinity between Sobolev spaces. Such boundedness was shown in the proof of \cite[Theorem~1.6(3)]{IshidaMerz2026}. We state this result in the following lemma and reproduce its proof.

\begin{lemma}
  \label{lem:AV-sandwich}
  Assume $\alpha\in[1,2]\cap(0,d)$, $\kappa<\kappa_{\rm c}$, and $V^L\in C^2(\R^d)$ obeys \eqref{ass5} for $\rho=1$ and all multi-indices $\beta\in\N_0^d$ with $|\beta|\leq2$. Then,
  \begin{align}
    \label{eq:AV-sandwich}
    \langle D\rangle^{-\alpha/2}AV^L \langle D\rangle^{-\alpha/2}\in\B(L^2).
  \end{align}
  Let $H_\kappa=L_\kappa-V^L$. If $\lambda+H_\kappa \geq \mu>0$ and $f\in C_c^\infty((0,\infty))$, then
  \begin{align}
    \label{eq:RAVf}
    (\lambda + H_\kappa)^{-1}AV^Lf(L_\kappa) \in \B(L^2).
  \end{align}
\end{lemma}

\begin{proof}
  Since $A=x\cdot D-id/2$,
  $$
  A V^L = \sum_{j=1}^d (x_jV^L)D_j -i\,x\cdot\nabla V^L-\frac{id}{2}V^L.
  $$
  All coefficients in this expression, together with the required derivatives, are bounded. Let $a_j=x_jV^L$ and write
  $$
  \langle D\rangle^{-\alpha/2}a_j D_j \langle D\rangle^{-\alpha/2}
  = \left(\langle D\rangle^{-\alpha/2}a_j \langle D\rangle^{\alpha/2}\right) \left(\langle D\rangle^{-\alpha/2}D_j \langle D\rangle^{-\alpha/2}\right).
  $$
  The first factor is bounded by the Lipschitz norm of $a_j$, and the second is bounded since $\alpha\geq1$. The zeroth-order terms are easier. This proves \eqref{eq:AV-sandwich}.
  Since $\kappa<\kappa_{\rm c}$,
  $$
  (\lambda + H_\kappa)^{-1}\langle D\rangle^{\alpha/2}\in\B(L^2),
  \qquad
  \langle D\rangle^{\alpha/2}f(L_\kappa)\in\B(L^2).
  $$
  Inserting these two factors on the sides of \eqref{eq:AV-sandwich} proves \eqref{eq:RAVf}.
\end{proof}

Let us also record that for $V^L\in C^2(\R^d)$ satisfying \eqref{ass5} with $\rho=1$ for all multi-indices $\beta\in\N_0^d$ with $|\beta|\leq2$, and $H_\kappa=L_\kappa-V^L$,
$$
[A,(\lambda+H_\kappa)^{-1}] = -(\lambda+H_\kappa)^{-1}[A,H_\kappa](\lambda+H_\kappa)^{-1} \in\B(L^2)
$$
since
$$
[A,H_\kappa]=i\alpha L_\kappa +ix\cdot\nabla V^L
$$
and $x\cdot\nabla V^L$ is bounded. Hence, for $f\in C_c^\infty$,
\begin{align}
  \label{eq:ARVf}
  A(\lambda+H_\kappa)^{-1}V^L f(L_\kappa) \in \B(L^2)
\end{align}
by Lemma~\ref{lem:AV-sandwich} and the boundedness of $V^L$.

%%%%%%%%%%%%%%%%%%

\subsection{Weighted estimates}
\label{s:weightedmaximalvelocity}

We first state a weighted resolvent estimate, which is useful for the Helffer--Sj\"ostrand functional calculus.

\begin{lemma}[{\cite[Lemma~6.6]{IshidaMerz2026}}]
  \label{lem8}
  Let $d\in\N$, $\alpha\in[1,2]\cap(0,d)$, $\kappa<\kappa_{\rm c}$, and $\rho>0$.
  Assume $V \in L^\infty(\R^d)$ satisfies $|V(x)| \lesssim \langle x\rangle^{-\rho}$.
  Let $H_\kappa=L_\kappa-V$, $z\in\C\setminus\R$, and $\mu\in[0,1]$.
  Then there exists a positive function $Z(z,\mu)$, which grows at most polynomially in $\langle z\rangle$ and $|\im(z)|^{-1}$, such that
  \begin{align}
    \label{eq_lem8_1}
    \|\langle x\rangle^\mu(z-H_\kappa)^{-1}\langle x\rangle^{-\mu}\|
    & \lesssim Z(z,\mu)
  \end{align}
  holds for $0\leq\mu\leq1$.
  Moreover, this statement also holds for $1<\mu\leq2$ whenever $\alpha\geq1$ if $d\geq3$ and $\alpha>1$ if $d=2$.   
\end{lemma}

Next, we recall the following weighted propagation estimate.

\begin{lemma}[{\cite[Lemma~6.7]{IshidaMerz2026}}]
  \label{lem9}
  Let $d\in\N$, $\alpha\in[1,2]\cap(0,d)$, $\kappa<\kappa_{\rm c}$, $\rho>0$, $\mu\in[0,1]$, and $f\in C_c^\infty(0,\infty)$.
  Assume $V \in L^\infty(\R^d)$ satisfies $|V(x)| \lesssim \langle x\rangle^{-\rho}$.
  Then, for $H_\kappa=L_\kappa-V$,
  \begin{equation}
    \label{eq_lem9_1}
    \|\langle x\rangle^\mu f(H_\kappa) \me{-i t H_\kappa}\langle x\rangle^{-\mu}\|
    \lesssim_f \langle t\rangle^\mu, \quad t\in\R.
  \end{equation}
  Moreover, \eqref{eq_lem9_1} also holds for $1<\mu\leq2$ for all $d\geq3$, or $d=2$ and $\alpha>1$.
\end{lemma}

We will also use the following refinement of Lemma~\ref{lem9}.

\begin{lemma}
  \label{lem:refined-first-moment}
  Let $\alpha\in[1,2]\cap(0,d)$, $\kappa<\kappa_{\rm c}$, and $V\in L^\infty(\R^d)$. Let $H_\kappa=L_\kappa-V$ and $f\in C_c^\infty(\R)$. Then,
  \begin{align}
    \label{eq:refined-first-moment}
    \|\langle x\rangle f(H_\kappa) \me{-itH_\kappa}\phi\|
    \leq C_{f,V} \left(\|\langle x\rangle\phi\|+|t|\|\phi\|\right),
    \quad t\in\R, \quad \phi\in\D(\langle x\rangle).
  \end{align}
\end{lemma}

\begin{proof}
  Set $P=|D|^\alpha$, $H=H_\kappa$, and $F=f(H)$. Since $\kappa<\kappa_{\rm c}$ and $V$ is bounded, $H$ is bounded from below and its form domain is $H^{\alpha/2}(\R^d)$. Since $\alpha-1\leq\alpha/2$,
  \begin{align}
    \label{eq:first-moment-spectral-regularity}
    \|Fu\|_{H^{\alpha-1}}\leq C_f\|u\|.
  \end{align}
  Choose $\eta\in C_c^\infty(\R^d;[0,1])$ equal to $1$ near the origin. For $L\geq1$ and $j\in\{1,...,d\}$, define
  $$
  a_{L,j}(x)=x_j\eta(x/L).
  $$

  \medskip
  \noindent
  {\bf Step 1: $a_{L,j}$ preserves $\D(H)$.}
  For a Schwartz function $a$, Plancherel's theorem and Young's inequality yield
  \begin{align}
    \label{eq:first-moment-Fourier-bound}
    \|[P,a]u\| \leq C_\alpha\left(\left\||\zeta|\widehat a(\zeta)\right\|_1\|u\|_{H^{\alpha-1}} + \left\||\zeta|^\alpha\widehat a(\zeta)\right\|_1\|u\|\right)
  \end{align}
  which follows from the following bound in Fourier space,
  \begin{align}
    \label{eq:P-commutator}
    \left||\xi|^\alpha-|\xi-\zeta|^\alpha\right|
    \leq C_\alpha|\zeta| \left(\langle\xi-\zeta\rangle^{\alpha-1} + |\zeta|^{\alpha-1}\right), \quad \alpha>0.
  \end{align}
  For $s\in\{1,\alpha\}$, we have $\left\||\zeta|^s\widehat a_{L,j}(\zeta)\right\|_1 = L^{1-s}\left\||\zeta|^s\widehat a_{1,j}(\zeta)\right\|_1$. Thus,
  \begin{align}
    \label{eq:coordinate-cutoff-commutator}
    \|[P,a_{L,j}]u\|\leq C\|u\|_{H^{\alpha-1}},
  \end{align}
  uniformly in $L$. Multiplication by $a_{L,j}$ preserves the form domain $H^{\alpha/2}(\R^d)$ for fixed $L$, and \eqref{eq:coordinate-cutoff-commutator} also shows that $a_{L,j}$ preserves $\D(H)$: indeed, for $u\in\D(H)$, the right-hand side of the form identity
  $$
  Ha_{L,j}u = a_{L,j}Hu + [P,a_{L,j}]u
  $$
  is in $L^2$, and the representation theorem for the closed form implies $a_{L,j}u\in\D(H)$.

  \medskip
  \noindent
  {\bf Step 2: Estimate for $\|a_{L,j}F\phi\|$.}
  Let $R(z)=(z-H)^{-1}$. For $z$ in a fixed compact set and $\im z\neq0$, $\alpha-1\leq\alpha/2$ and the spectral theorem yield
  $$
  \|R(z)\|_{L^2\to H^{\alpha-1}} \leq C|\im z|^{-1}.
  $$
  Therefore,
  $$
  [a_{L,j},R(z)] = R(z)[a_{L,j},H]R(z),
  \qquad
  \|[a_{L,j},R(z)]\| \leq C|\im z|^{-2},
  $$
  uniformly in $L$. Thus, by the Helffer--Sj\"ostrand formula with an almost-analytic extension of sufficiently high order,
  \begin{align*}
    \sup_{L\geq1}\|[a_{L,j},F]\|\leq C_f.
  \end{align*}
  Thus, for $\phi\in\D(\langle x\rangle)$,
  \begin{align}
    \label{eq:position-cutoff-spectral-cutoff}
    \|a_{L,j}F\phi\| \leq \|f\|_\infty\|x_j\phi\|+C_f\|\phi\|.
  \end{align}
  
  \medskip
  \noindent
  {\bf Step 3: Estimate for $\|a_{L,j}FU(t)\phi\|$.}
  By \eqref{eq:first-moment-spectral-regularity} and \eqref{eq:coordinate-cutoff-commutator},
  \begin{align}
    \label{eq:energy-commutator-spectrally-localized}
    \sup_{L\geq1}\|[H,a_{L,j}]F\|\leq C_f.
  \end{align}
  By Duhamel's formula and Step~1,
  \begin{align*}
    a_{L,j}\me{-itH}F\phi
    = \me{-itH}a_{L,j}F\phi + i\int_0^t\me{-i(t-s)H}[H,a_{L,j}]F\me{-isH}\phi\,ds
  \end{align*}
  holds for all $\phi\in L^2$. Thus, by \eqref{eq:position-cutoff-spectral-cutoff} and \eqref{eq:energy-commutator-spectrally-localized},
  \begin{align*}
    \|a_{L,j}F\me{-itH}\phi\| \leq \|f\|_\infty\|x_j\phi\| + C_f(1+|t|)\|\phi\|.
  \end{align*}

  \medskip
  \noindent
  {\bf Step 4: Conclusion.}
  Letting $L\to\infty$ and using Fatou's lemma proves the same estimate with $a_{L,j}$ replaced with $x_j$. Summing over $j$, adding the $L^2$ bound for $F\me{-itH}$, and using $\|\phi\|\leq\|\langle x\rangle\phi\|$ proves \eqref{eq:refined-first-moment}.
\end{proof}

The proof of Theorem~\ref{asympcomplhardyplusperturbationnonexistence} also rests on the following minimal escape velocity estimate which we obtained as a corollary of abstract results in \cite{Skibsted1991}. In the following, $\sigma_{\rm p}(H)$ denotes the point spectrum of a self-adjoint operator $H$.

\begin{theorem}[{\cite[Theorem~1.5]{IshidaMerz2026}}]
  \label{the6}
  Let $d\in\N$, $\alpha\in(0,2]\cap(0,d)$, $\kappa<\kappa_{\rm c}$, and $\rho>0$.
  Assume $V\in C^\infty(\R^d)$ satisfies \eqref{ass5} for all multi-indices $\beta$.
  Let $H_\kappa=L_\kappa-V$ and $f\in C_c^\infty(0,\infty)$ with $\supp f \cap \sigma_{\rm p}(H_\kappa) = \emptyset$.
  Then, for all $\nu\geq1$ there are $\theta>0$ and $T>0$ such that
  \begin{equation}
    \label{eq_the6_1}
    \left\| \one\left(\frac{A}{t}\leq\theta\right)f(H_\kappa) \me{-i t H_\kappa}\phi\right\|
    \lesssim_{\theta,\nu,f,T,V} t^{-\nu}\|\langle A\rangle^\nu\phi\|, \quad t\geq T,
  \end{equation}
  holds for all $\phi\in\D(\langle A\rangle^\nu)$.
\end{theorem}

Our proof of Theorem~\ref{asympcomplhardyplusperturbationnonexistence} uses Theorem~\ref{the6} with $V=0$.

%%%%%%%%%%%%%%%%%%%%%%%%%%%%%%%%%%%%%%%%
%%%%%%%%%%%%%%%%%%%%%%%%%%%%%%%%%%%%%%%%

\section{Proof of (\ref{eq:the7}) in Theorem~\ref{the7}}
\label{s:maximalvelocity}

In this section, we prove the time-integrated maximal velocity estimate \eqref{eq:the7} in Theorem~\ref{the7}.
First note that by decomposing $f$ into its real and imaginary parts, it suffices to consider real-valued $f\in C_c^\infty(0,\infty)$, which we assume from now on.

The proof principally relies on the analysis of the time-evolution of a suitable propagation observable: It should measure whether $|x|/t$ is above a certain threshold for energies belonging to a given interval. To make this precise, we need some preparation. Fix $q\in C_c^\infty((1/2,4);[0,1])$ such that
$$
q(r)=1 \quad\text{for }1\leq r\leq2.
$$
For $n\in\N$, set $R_n:=2^n$
and define
\begin{align}
  \chi_n(s) :=
  \begin{cases}
    q\left(\sqrt{s}/R_n\right), & s\geq0, \\
    0, & s<0.
  \end{cases}
\end{align}
Since $q$ vanishes in a neighborhood of the origin, $\chi_n\in C_c^\infty(\R_+)$. Moreover,
\begin{align}
  \label{eq:dyadic-cutoff-plateau}
  \chi_n(s)=1
  \quad\text{if}\quad
  R_n^2\leq s\leq R_{n+1}^2
\end{align}
and
\begin{align}
  \label{eq:dyadic-cutoff-support}
  \supp\chi_n \subset \left[\frac{R_n^2}{4},16R_n^2\right].
\end{align}
Equivalently,
\begin{align}
  \label{eq:dyadic-spatial-cutoff}
  \chi_n\left(\frac{|x|^2}{t^2}\right) = q\left(\frac{|x|}{R_n t}\right).
\end{align}
In particular, for every multi-index $\beta\in\N_0^d$,
\begin{align}
  \label{eq:dyadic-spatial-derivatives}
  \left\| \partial_x^\beta \chi_n\left(\frac{|x|^2}{t^2}\right) \right\|_\infty \lesssim_\beta
  (R_n t)^{-|\beta|},
\end{align}
uniformly in $n$ and $t>0$.
Thus, with $\chi_n$ we measure whether $|x|/t$ belongs to a dyadic shell at scale $R_n$. Next, we introduce the switch function $\R\ni s\mapsto X_n(s)$, defined by
$$
X_n(s) = \int_0^s\chi_n(r)^2\,dr \quad \text{for} \ s\geq0,
\qquad
X_n(s) = 0 \quad \text{for} \ s<0.
$$
It measures whether $|x|/t$ is above the threshold $R_n$. We have
$$
0\leq X_n(s)\leq s \quad \text{for} \ s\geq0,
\qquad
X_n'(s)=\chi_n(s)^2,
$$
and
\begin{align}
  \label{eq:dyadic-X-derivatives}
  \|X_n^{(k)}\|_\infty \lesssim_k R_n^{2-2k}, \quad k\geq1.
\end{align}
We now introduce the relevant propagation observable
\begin{align}
  \Phi_n(t) := F X_n\left(B_t\right) F, 
\end{align}
where
$$
F := f(H_\kappa),
\qquad
B_t := \frac{|x|^2}{t^2}.
$$

One of the  main difficulties here, compared with the analysis of local Schr\"odinger operator like $-\Delta+V$, is the nonlocality of $|D|^\alpha$ when $\alpha<2$. Indeed, for $\alpha=2$, commutators of $-\Delta$ with a non-compactly supported switch function (like $X_n$ here) involve only derivatives of the switch function. Thus, even if the switch is one at  infinity, its commutator with $-\Delta$ is supported in the bounded transition region of the switch. This allows to directly control the whole region $|x|/t\geq\Theta$, without introducing an upper cutoff. For $\alpha<2$, this localization property is lost since $[|D|^\alpha,X_n]$ or $[|D|^\alpha,\chi_n^2]$ are nonlocal, hence the remainder terms arising in the commutator expansions require more care.
The Hardy potential, while only being singular at the origin, also still needs to be treated carefully. Furthermore, the commutator expansions generate Fourier multipliers of the order $\alpha-1$, which explains why the case $\alpha<1$ is even more difficult and requires a separate low-frequency analysis.
The decomposition of the outgoing region into dyadic shells and the energy localization allow us to obtain estimates which are summable in $n$ and prove the maximal velocity estimate still without upper cutoff.

%%%%%%%%%%%%%%%%%%

\subsection{Preparations to control the time evolution of $\Phi_n(t)$}

We control the time evolution of $\Phi_n(t)$ by considering its Heisenberg derivative
\begin{align}
  \mathbb{D}_{H_\kappa} \Phi_n(t)
  = \frac{\partial \Phi_n(t)}{\partial t} + i[H_\kappa, \Phi_n(t)].
\end{align}
In the following, commutators containing $H_\kappa$ or unbounded Fourier multipliers are interpreted first as quadratic forms. Since $\kappa<\kappa_{\rm c}$ and $V$ is bounded, the form domain of $H_\kappa$ is $H^{\alpha/2}(\R^d)$ and $C_c^\infty(\R^d\setminus\{0\})$ is a form core. All commutator identities below are first established on $C_c^\infty(\R^d\setminus\{0\})$.

The function $X_n$ is bounded and smooth, and all its derivatives are compactly supported. However, since $X_n$ is not compactly supported, we cannot use the ordinary Helffer--Sj\"ostrand formula \eqref{eq:hs}; likewise the version in \cite[Lemma~5.3]{IshidaMerz2026} which applies to $S^0(\R)$-symbols is unavailable since it requires that the commutator is taken with a bounded operator, which is not the case here. Thus, to proceed, we note that the commutator $[|D|^\alpha,X_n(B_t)]$ does not change if we subtract a constant from $X_n(B_t)$. Hence, instead of $X_n(B_t)$, we are going to consider $Y_n(B_t)$ with
\begin{align}
  Y_n(s) := X_n(s)-c_n\eta_n(s),
\end{align}
where $c_n:=\lim_{s\to\infty}X_n(s)=\int_0^\infty\chi_n(r)^2\,dr$, $\eta_n(s)=\eta_*(s/R_n^2)$, and $\eta_*\in C^\infty(\R;[0,1])$ is such that
$$
\eta_*(s)=0\quad\text{for}\ s\leq-2,
\qquad\text{and}\qquad
\eta_*(s)=1\quad\text{for}\ s\geq-1.
$$
Since $X_n$ vanishes on $(-\infty,0]$ and is constant for sufficiently large positive $s$, we have $Y_n \in C_c^\infty(\R)$. Moreover, with $X_*(s)=R_n^{-2}X_n(R_n^2s)$ and $c_*:=R_n^{-2}c_n$, we see
$$
Y_n(s)=R_n^2Y_*(s/R_n^2),
\qquad
Y_*(s):=X_*(s)-c_*\eta_*(s),
$$
where $Y_*$ is independent of $n$. Since $B_t=|x|^2/t^2\geq0$,
\begin{align}
  \label{eq:Yn-Xn-exact-repair}
  \eta_n(B_t) = 1,
  \qquad
  X_n(B_t) = Y_n(B_t) + c_n,
  \qquad
  Y_n'(B_t) = X_n'(B_t) = \chi_n(B_t)^2.
\end{align}
Thus,
$$
F[H_\kappa,X_n(B_t)]F = F[H_\kappa,Y_n(B_t)]F
$$
and, likewise,
$$
\partial_t X_n(B_t) = \partial_t Y_n(B_t).
$$
Thus, it is just $Y_n$ that enters into the Helffer--Sj\"ostrand calculus. To apply it, let $\widetilde Y_{*,N}$ be a compactly supported almost-analytic extension of $Y_*$ of order $N$, and set
$$
\widetilde Y_{n,N}(z) := R_n^2 \widetilde Y_{*,N}(z/R_n^2).
$$
To work with the Helffer--Sj\"ostrand formula, we need good control on $\widetilde Y_{n,N}(z)$. First,
$$
|\bar\partial\widetilde Y_{*,N}(u+iv)| \lesssim_N \langle u\rangle^{-N-1}|v|^N \one_{|v|\leq2\langle u\rangle},
$$
and therefore,
\begin{align}
  \label{eq_the7_3}
  \begin{split}
  |\bar\partial\widetilde Y_{n,N}(z)|
  & \lesssim_N \left|\frac{\im z}{R_n^2}\right|^N \left\langle\frac{\re z}{R_n^2}\right\rangle^{-N-1} \cdot \one\left( |\im z| \leq 2R_n^2 \left\langle\frac{\re z}{R_n^2}\right\rangle \right).
  \end{split}
\end{align}
Thus, for $k\in\N_0$ and $k+1<m \leq N$, we obtain
\begin{align}
  \label{eq:Yn-AA}
  \int_{\mathbb C} |\overline\partial\widetilde Y_{n,N}(z)| |\im z|^{-m} \langle z\rangle^k\,d^2z
  \lesssim_{N,m,k} R_n^{4-2m+2k}
\end{align}
with $2d^2z$ being the total variation of $dz\,\wedge d\overline z$.
Introducing the unscaled functions
\begin{align}
  \chi_*(s) :=
  \begin{cases}
    q(\sqrt{s}), & s\geq0,\\
    0, & s<0,
  \end{cases}
\end{align}
we can write
\begin{align}
  \label{eq:dyadic-X-scaling}
  \chi_n(s) = \chi_*\left(\frac{s}{R_n^2}\right).
\end{align}
Similarly as before, choose a compactly supported almost analytic extension $\widetilde\chi_{*,N}$ of $\chi_*$ and define
$$
\widetilde\chi_{n,N}(z) := \widetilde\chi_{*,N}\left(\frac{z}{R_n^2}\right).
$$
Then, for $k\in\N_0$ and $k+1<m\leq N$,
\begin{align}
  \label{eq:dyadic-chi-AA-integral}
  \int_{\C} |\bar\partial\widetilde\chi_{n,N}(z)| |\im z|^{-m} \langle z\rangle^k \,d^2z
  \lesssim_{N,m} R_n^{2-2m +2k}.
\end{align}

To apply the Helffer--Sj\"ostrand formula for $F[H_\kappa,Y_n(B_t)]F$, note that since $f$ is compactly supported, $H_\kappa F=FH_\kappa\in\mathcal B(L^2)$. For $M\in\mathcal B(L^2)$, we use the notation
\begin{align}
  \label{eq:localized-H-commutator-the7aux2}
  F[H_\kappa,M]F := (H_\kappa F)MF-FM(H_\kappa F).
\end{align}
This defines a bounded operator without assuming that $M$ preserves $\D(H_\kappa)$. When both $M$ and $M^*$ preserve the form domain, it agrees with the spectrally localized form commutator. Thus, with $R_t(z) := (z-B_t)^{-1}$, we have for the integrand in the Helffer--Sj\"ostrand formula,
\begin{align}
  \label{eq:sandwiched-resolvent-commutator-the7aux2}
  F[H_\kappa,R_t(z)]F = F R_t(z)[|D|^\alpha,B_t]R_t(z)F
\end{align}
in the sense of quadratic forms.
Indeed, $R_t(z)$ and $B_tR_t(z)$, together with their adjoints, preserve $H^{\alpha/2}(\R^d)$: these multipliers and their first derivatives are bounded for fixed $z\notin\R$ and $t>0$. Let $\mathfrak h$ denote the closed form of $H_\kappa$. Then, for $u,v\in H^{\alpha/2}(\R^d)$,
\begin{align*}
  \mathfrak h[R_t(\bar z)u,B_tR_t(z)v] - \mathfrak h[B_tR_t(\bar z)u,R_t(z)v] 
  = \mathfrak h[u,R_t(z)v] -\mathfrak h[R_t(\bar z)u,v],
\end{align*}
by $B_tR_t(z)=zR_t(z)-1$. The potential contributions on the left cancel. Sandwiching with $F$ therefore gives \eqref{eq:sandwiched-resolvent-commutator-the7aux2}. Moreover,
$$
\|F[H_\kappa,R_t(z)]F\| \leq 2\|H_\kappa F\|\,\|f\|_\infty\,|\im z|^{-1}.
$$
Since $\widetilde Y_{n,N}$ is compactly supported and $N$ can be chosen sufficiently large, the Helffer--Sj\"ostrand formula \eqref{eq:hs} may now be used and yields
\begin{align}
  \label{eq:HS-Yn}
  F[H_\kappa,Y_n(B_t)]F
  & = \frac{1}{2\pi i} \int_{\mathbb C} \overline\partial\widetilde Y_{n,N}(z)\, F R_t(z)[|D|^\alpha,B_t]R_t(z)F\, dz\wedge d\overline z,
\end{align}
where the localized commutator integral converges absolutely in operator norm.

Note also that $Y_n(B_t)$, $\chi_n(B_t)$, and their first spatial derivatives preserve $H^{\alpha/2}(\R^d)$ since $\alpha\leq2$ and the corresponding multipliers belong to $W^{1,\infty}$. Moreover, for each fixed $n$, the map $t\mapsto X_n(B_t)$ is continuously differentiable in operator norm, with $\partial_t X_n(B_t) = -\frac{2}{t}B_t\chi_n(B_t)^2$. Since $f$ is compactly supported, $t\mapsto f(H_\kappa)\me{-itH_\kappa}$ is also continuously differentiable in operator norm. Thus, the Heisenberg derivative of
$$
\Phi_n(t) = FX_n(B_t)F = FY_n(B_t)F + c_nF^2
$$
and the associated commutator calculations below are first understood as a bounded linear operator from $H^{\alpha/2}(\R^d)$ to $H^{-\alpha/2}(\R^d)$. The identity for the Heisenberg derivative follows from ordinary differentiation for every initial vector in $L^2$, and has the bounded representative
$$
i(H_\kappa F)X_n(B_t)F - iF X_n(B_t)(H_\kappa F) - \frac{2}{t}F B_t\chi_n(B_t)^2F.
$$
We now estimate the Heisenberg derivative of $\Phi_n(t)$.

\begin{lemma}
  \label{the7aux2}
  Let $d\in\N$, $\alpha\in[1,2]\cap(2-d/2,d)$, $\kappa<\kappa_{\rm c}$, and assume the hypotheses on $V$ imposed in Theorem~\ref{the7}. Let $B_t$, $F$, and $\Phi_n(t)$ be defined as above. Then,
  \begin{align}
    \label{eq:the7aux2}
    \begin{split}
      \mathbb D_{H_\kappa}\Phi_n(t)
      & = F \chi_n(B_t) \left[ \frac{\alpha}{t^2} \left( x\cdot|D|^{\alpha-2}D + |D|^{\alpha-2}D\cdot x \right) - \frac{2|x|^2}{t^3} \right] \chi_n(B_t)F \\
      & \quad + \ce_{1,n}(t) + \ce_{2,n}(t)
    \end{split}
  \end{align}
  as a quadratic form identity on $H^{\alpha/2}(\R^d)$, where
  \begin{align}
    \ce_{1,n}(t) & := \frac1{2\pi} F \left(\int_{\C} \bar\partial\widetilde Y_{n,N}(z) (z-B_t)^{-2} \,\ad_{B_t}^2(|D|^\alpha) (z-B_t)^{-1} \,dz\wedge d\bar z \right) F, \\
    \ce_{2,n}(t) & := i F \chi_n(B_t) [\chi_n(B_t),[|D|^\alpha,B_t]] F,
  \end{align}
  with $\widetilde Y_{n,N}(z)$ being an almost analytic extension of $Y_n$ of sufficiently large order $N$.
  The quadratic forms corresponding to $\me{itH_\kappa}\ce_{1,n}(t)\me{-itH_\kappa}$ and $\me{itH_\kappa}\ce_{2,n}(t)\me{-itH_\kappa}$ are estimated in Lemma~\ref{boundr12}.
\end{lemma}

We postpone the proof of this lemma after the proof of Theorem~\ref{the7}.
In the next lemma, we estimate the principal term in \eqref{eq:the7aux2}.

\begin{lemma}
  \label{the7aux3}
  Let $d\in\N$, $\alpha\in[1,2]$, $\kappa<\kappa_{\rm c}$, and assume the hypotheses on $V$ imposed in Theorem~\ref{the7}. Let $B_t$, $F$, and $\Phi_n(t)$ be defined as above. Then, there exists $n_0=n_0(d,\alpha,\kappa,\rho,f,V)\in\N$ such that, for all $n\geq n_0$,
  \begin{align}
    \begin{split}
      & F\chi_n(B_t) \left[ \frac{\alpha}{t^2} \left( x\cdot|D|^{\alpha-2}D + |D|^{\alpha-2}D\cdot x \right) - \frac{2|x|^2}{t^3} \right] \chi_n(B_t)F \\
      & \quad \leq -\frac{cR_n^2}{t} F \chi_n(B_t)^2 F + \frac{C}{t^2} + \ce_{3,n}(t),
    \end{split}
  \end{align}
  where $c,C>0$ are independent of $n\geq n_0$, and
  \begin{equation}
    \ce_{3,n}(t) := \frac{\alpha}{t^2} F [G,\chi_n(B_t)] x\cdot|D|^{\alpha-2}D \chi_n(B_t) F + \mathrm{h.c.},
  \end{equation}
  with $G:=g(H_\kappa)$ and $g\in C_c^\infty((0,\infty);\R)$ being equal to $1$ on a neighborhood of $\supp f$. The quadratic form corresponding to $\me{itH_\kappa}\ce_{3,n}(t)\me{-itH_\kappa}$ is estimated in Lemma~\ref{boundr3}.
\end{lemma}

We postpone the proof of this lemma after the proof of Lemma~\ref{the7aux2}. At this point, let us note that the quadratic forms corresponding to $\me{itH_\kappa}\ce_{j,n}(t)\me{-itH_\kappa}$, $j\in\{1,2,3\}$, are bounded in terms of $\|\langle x\rangle^2\phi\|$. In order to obtain estimates involving only logarithmic moments, we will use the following lemma.
It states that for large start times $S \geq 1$, the propagation integral in the region $\{|x| \geq \Theta t\}$ beyond a threshold speed $\Theta_0$ is bounded by a constant plus terms decaying algebraically in $S$ for initial data in $\D(\langle x\rangle^2)$.

\begin{lemma}
  \label{lem:late-start}
  Assume the hypotheses of Theorem~\ref{the7}. Set
  \begin{align}
    \label{eq:integrable-remainder-exponent}
    \beta =
    \begin{cases}
      \alpha, & \quad \alpha>1, \\
      \min\{2,1+\rho\}, & \quad \alpha=1 \text{ and } V\neq0, \\
      2, & \quad \alpha = 1 \text{ and } V=0.
    \end{cases}
  \end{align}
  Then, $\beta\in(1,2]$ and there exists $\Theta_0=\Theta_0(d,\alpha,\kappa,\rho,f,V)>0$ such that for all $S\geq1$, $\Theta\geq\Theta_0$, and $\phi\in\D(\langle x\rangle^2)$,
  \begin{align}
    \label{eq:late-start-bound}
    \begin{split}
      \int_S^\infty \|\one_{\{|x|/t\geq\Theta\}}f(H_\kappa) \me{-itH_\kappa}\phi\|^2\frac{dt}{t}
      & \leq C\left(\|\phi\|^2 + S^{-2}\|\langle x\rangle\phi\|^2 + S^{1-\beta}\|\langle x\rangle^2\phi\|^2\right),
    \end{split}
  \end{align}
  where $C$ is independent of $S$, $\Theta$, and $\phi$.
\end{lemma}

Since $\beta>1$, the terms $S^{-2}\|\langle x\rangle\phi\|$ and $S^{1-\beta}\|\langle x\rangle^2\phi\|$ in \eqref{eq:late-start-bound} vanish as $S\to\infty$. Balancing $S$ against a suitable spatial cutoff will allow us to prove \eqref{eq:the7} for $\phi \in \D\left((\log(2+\langle x\rangle))^\gamma\right)$ (instead of for just $\phi\in\D(\langle x\rangle^2)$).
We postpone the proof of this lemma after Lemma~\ref{lem:xK}.

%%%%%%%%%%%%%%%%%%

\subsection{Proof of (\ref{eq:the7}) in Theorem~\ref{the7}}

We now give the proof of \eqref{eq:the7} in Theorem~\ref{the7}.
Write $H=H_\kappa$, $F=f(H)$, and $U(t)=\me{-itH}$. Let $\Theta_0$ be as in Lemma~\ref{lem:late-start}. In the following, we estimate
$$
(\mathcal T\phi)(t) := \one_{\{|x|/t\geq\Theta_0\}}FU(t)\phi, \quad t\geq1,
$$
in $L^2([1,\infty),dt/t;L^2(\R^d))$.

For $j\in\N_0$, we partition $\R^d$ dyadically using the sets $E_j=\{x:2^j\leq\langle x\rangle<2^{j+1}\}$. For $\phi\in L^2(\R^d)$, let $\phi_j=\one_{E_j}\phi$. We have
\begin{align}
  \label{eq:initial-shell-moments}
  \|\langle x\rangle^m\phi_j\| \leq 2^{m(j+1)}\|\phi_j\|, \quad m\geq0.
\end{align}
Let $\beta$ be as in \eqref{eq:integrable-remainder-exponent} and set
$$
p=\frac4{\beta-1}\geq4,
\qquad
S_j=2^{p(j+1)}.
$$
With this choice, we can control $S_j^{-2}2^{2(j+1)}\leq1$ and $S_j^{1-\beta}2^{4(j+1)}=1$. Thus, by Lemma~\ref{lem:late-start} applied to $\phi_j$,
\begin{align}
  \label{eq:initial-shell-late-time}
  \int_{S_j}^\infty\|\mathcal T\phi_j(t)\|^2\frac{dt}{t}
  \lesssim \|\phi_j\|^2.
\end{align}
Fix $\gamma>1/2$ and let $\phi\in\D(\log(2+\langle x\rangle)^\gamma)$.
For a finite partial sum $\phi^{(J)}=\sum_{j=0}^J\phi_j$, write
$$
\mathcal{T}\varphi^{(J)}(t) = \mathcal{U}_J(t) + \mathcal{V}_J(t),
$$
where
\begin{align*}
  \mathcal{U}_J(t) = \sum_{0\leq j\leq J,\, t<S_j}\mathcal{T}\phi_j(t),
  \qquad
  \mathcal{V}_J(t) = \sum_{0\leq j\leq J,\, t\geq S_j}\mathcal{T}\phi_j(t).
\end{align*}
Since $\|\mathcal{T}\| \leq \|f\|_\infty$ uniformly in $t$, and the elements in $\{\varphi_j\}$ are mutually orthogonal,
\begin{equation}
  \|\mathcal{U}_J(t)\|^2
  \leq \|f\|_\infty^2 \left\|\sum_{0\leq j\leq J,\, t<S_j}\phi_j\right\|^2
  = \|f\|_\infty^2\sum_{0\leq j\leq J,\, t<S_j} \|\phi_j\|^2.
\end{equation}
Thus,
\begin{equation}
  \int_1^\infty\|\mathcal{U}_J(t)\|^2 \frac{dt}{t}
  \leq \|f\|_\infty^2 \sum_{j=0}^J \|\phi_j\|^2 \int_1^{S_j}\frac{dt}{t}
  = p\, (\log2)\, \|f\|_\infty^2 \sum_{j=0}^J (j+1)\|\phi_j\|^2
\end{equation}
and hence
\begin{equation}
  \label{eq:bound-UJ}
  \|\mathcal{U}_J\|_{L^2(dt/t;L^2)}
  \lesssim_f \left(\sum_{j=0}^J(j+1)\|\phi_j\|^2\right)^{1/2}.
\end{equation}
By \eqref{eq:initial-shell-late-time} and Cauchy--Schwarz,
\begin{align}
  \label{eq:bound-VJ}
  \begin{split}
    \|\mathcal{V}_J\|_{L^2(dt/t;L^2)}
    & \leq\sum_{j=0}^J\left(\int_{S_j}^\infty \|\mathcal{T}\phi_j(t)\|^2 \frac{dt}{t}\right)^{1/2} 
      \lesssim \sum_{j=0}^J\|\phi_j\| \\
    & \lesssim \left(\sum_{j=0}^\infty(j+1)^{-2\gamma}\right)^{1/2}\left(\sum_{j=0}^J (j+1)^{2\gamma}\|\phi_j\|^2\right)^{1/2} \\
    & \lesssim_\gamma \left(\sum_{j=0}^J(j+1)^{2\gamma}\|\phi_j\|^2\right)^{1/2}
  \end{split}
\end{align}
since $\gamma>1/2$. By \eqref{eq:bound-UJ} and \eqref{eq:bound-VJ},
\begin{equation}
  \|\mathcal{T}\phi^{(J)}\|_{L^2(dt/t;L^2)}
  \lesssim_\gamma \left(\sum_{j=0}^J (j+1)^{2\gamma} \|\phi_j\|^2\right)^{1/2}
  \lesssim_\gamma \|(\log(2+\langle x\rangle))^\gamma \phi\|.
\end{equation}
In the last inequality we used $\log(2+\langle x\rangle)\sim j+1$ on $E_j$, uniformly in $j$.

\iffalse
  Minkowski's inequality and \eqref{eq:initial-shell-integrated-bound} yield
  \begin{align*}
    \|\mathcal T\phi^{(J)}\|_{L^2(dt/t;L^2)}
    & \lesssim \sum_{j=0}^J (j+1)^{1/2} \|\phi_j\| \\
    & \lesssim \left(\sum_{j\geq0}(j+1)^{1-2\gamma}\right)^{1/2} \left(\sum_{j=0}^J(j+1)^{2\gamma}\|\phi_j\|^2\right)^{1/2} \\
    & \lesssim_\gamma \left\|\left(\log(2+\langle x\rangle)\right)^\gamma \phi \right\|
  \end{align*}
  since $\gamma>1$. In the last inequality we used $\log(2+\langle x\rangle)\sim j+1$ on $E_j$, with constants independent of $j$.
\fi

Finally, since $\phi^{(J)}\to\phi$ in $L^2$ and $\|\one_{\{|x|/t\geq\Theta_0\}}FU(t)\|\leq\|f\|_\infty$, we have $\mathcal T\phi^{(J)}(t)\to\mathcal T\phi(t)$ in $L^2$ for every $t\geq1$. Fatou's lemma thus shows \eqref{eq:the7} for $\Theta=\Theta_0$. The same estimate holds for all $\Theta\geq\Theta_0$ by monotonicity.
\qed

\smallskip
Note that a simple application of Minkowski's inequality to estimate $\|\mathcal T\phi^{(J)}\|$ would have required us to take $\gamma>1$. Splitting the evolution into early and late contributions and using uniform boundedness of $\mathcal T$ and orthogonality of the not time-evolved $\phi_j$ allow us to control the ``early contribution'' $t<S_j$ more efficiently.

%%%%%%%%%%%%%%%%%%

\subsection{Proofs of Lemmas~\ref{the7aux2}--\ref{boundr3}}

It remains to prove Lemmas~\ref{the7aux2}--\ref{boundr3}. We begin with the

\begin{proof}[Proof of Lemma~\ref{the7aux2}]
  Recall $F=f(H_\kappa)$ and $B_t=|x|^2/t^2$.
  We write
  $$
  H := H_\kappa,
  \quad
  \chi_t := \chi_n(B_t),
  \quad
  R_t(z) := (z-B_t)^{-1}.
  $$
  We establish the identity as a (spectrally localized) quadratic form identity; the estimates on the remainder forms are proved in Lemma~\ref{boundr12} below.

  \medskip
  \noindent
  {\bf Step 1: Time derivative.}
  As $F$ is independent of $t$ and commutes with $H$,
  \begin{align}
    \label{eq:heisenberg-Phi-first}
    \mathbb D_H \Phi_n(t)
    = F\frac{\partial}{\partial t}X_n(B_t)F + i F[H,X_n(B_t)]F.
  \end{align}
  Since $X_n'(s)=\chi_n(s)^2$, the chain rule yields
  \begin{align}
    \label{eq:time-derivative-X}
    \frac{\partial}{\partial t} X_n(B_t) = -\frac{2|x|^2}{t^3}X_n'(B_t) = -\frac{2|x|^2}{t^3}\chi_t^2.
  \end{align}

  \medskip
  \noindent
  {\bf Step 2: Expansion of $F[|D|^\alpha,B_t]F$.}
  Let
  $$
  C_t := [|D|^\alpha,B_t].
  $$
  The following identities are first understood in the sense of quadratic forms on a common form core, such as $C_c^\infty(\R^d\setminus\{0\})$. The extensions of the final forms are justified by Lemma~\ref{boundr12}.

  First, since $[C_t,R_t(z)] = R_t(z)[C_t,B_t]R_t(z)$, we have
  \begin{align}
    \label{eq:move-C-through-resolvent}
    R_t(z)C_tR_t(z) = R_t(z)^2C_t + R_t(z)^2[C_t,B_t]R_t(z)
  \end{align}
  for the integrand in the Helffer--Sj\"ostrand formula. Moreover,
  $$
  [C_t,B_t] = [[|D|^\alpha,B_t],B_t] = \operatorname{ad}_{B_t}^2(|D|^\alpha).
  $$
  Substituting \eqref{eq:move-C-through-resolvent} into \eqref{eq:HS-Yn} yields
  \begin{align}
    \label{eq:Xnu-commutator-expansion}
    \begin{split}
      F[H,Y_n(B_t)]F
      & = F\left( \frac{1}{2\pi i} \int_{\C} \overline\partial\widetilde Y_{n,N}(z) R_t(z)^2 \,dz\wedge d\overline z \right) C_tF \\
      & \quad + \frac{1}{2\pi i} F\left( \int_{\C} \overline\partial\widetilde Y_{n,N}(z) R_t(z)^2 \operatorname{ad}_{B_t}^2(|D|^\alpha) R_t(z) \,dz\wedge d\overline z
      \right)F.
    \end{split}
  \end{align}
  By the Helffer--Sj\"ostrand formula,
  \begin{align}
    \chi_n(B_t)^2
    = Y_n'(B_t)
    = \frac{1}{2\pi i} \int_{\mathbb C} \overline\partial\widetilde Y_{n,N}(z)R_t(z)^2\, dz\wedge d\overline z.
  \end{align}
  Hence,
  \begin{align}
    \label{eq:i-Xnu-commutator}
    iF[H,Y_n(B_t)]F = i \, F\chi_t^2 C_t F + \ce_{1,n}(t),
  \end{align}
  where
  \begin{align}
    \ce_{1,n}(t) := \frac{1}{2\pi} F\left( \int_{\C} \overline\partial\widetilde Y_{n,N}(z) (z-B_t)^{-2}\,\operatorname{ad}_{B_t}^2(|D|^\alpha) (z-B_t)^{-1} \,dz\wedge d\overline z \right)F.
  \end{align}

  \medskip
  \noindent
  {\bf Step 3: Symmetrization of the principal term.}
  Next, we symmetrize the main term $iF\chi_t^2C_tF$. We have
  \begin{align}
    \label{eq:symmetrize-principal}
    i\chi_t^2 C_t = i \chi_t C_t \chi_t + i \chi_t[\chi_t,C_t].
  \end{align}
  The second term gives, after sandwiching with $F$,
  \begin{align}
    \ce_{2,n}(t) := iF\chi_t [\chi_t,[|D|^\alpha,B_t]] F.
  \end{align}
  It remains to calculate $i\chi_t C_t\chi_t$. Since $[|D|^\alpha,x_j] = -i\alpha|D|^{\alpha-2}D_j$, we obtain
  \begin{align}
    \label{eq:explicit-kinetic-commutator}
    i[|D|^\alpha,B_t]
    = \frac{\alpha}{t^2} \left( x\cdot|D|^{\alpha-2}D + |D|^{\alpha-2}D\cdot x \right).
  \end{align}

  \smallskip
  Combining \eqref{eq:heisenberg-Phi-first}, \eqref{eq:time-derivative-X}, \eqref{eq:i-Xnu-commutator}, \eqref{eq:symmetrize-principal}, and \eqref{eq:explicit-kinetic-commutator}, and using that $\chi_t$ commutes with $|x|^2$, we obtain
  \begin{align}
    \begin{split}
      \mathbb D_{H_\kappa} \Phi_n(t)
      & = F\chi_t \left[ \frac{\alpha}{t^2} \left( x\cdot|D|^{\alpha-2}D + |D|^{\alpha-2}D\cdot x \right) - \frac{2|x|^2}{t^3} \right] \chi_t F \\
      & \quad + \ce_{1,n}(t) + \ce_{2,n}(t).
    \end{split}
  \end{align}
  This proves \eqref{eq:the7aux2} as a quadratic form identity. Lemma~\ref{boundr12} below provides estimates for $\ce_{1,n}(t)$ and $\ce_{2,n}(t)$ and, in particular, justifies the extensions claimed above.
\end{proof}

\begin{proof}[Proof of Lemma~\ref{the7aux3}]
  In this proof, we merely isolate and estimate the leading term from the error term $\ce_{3,n}(t)$, which we estimate in Lemma~\ref{boundr3} below.

  \medskip
  \noindent
  {\bf Step 1: Setup and isolation of the leading term.}
  Recall $F=f(H_\kappa)$ and $G:=g(H_\kappa)$, where $g\in C_c^\infty((0,\infty);\R)$ equals $1$ on a neighborhood of $\supp f$. Thus, $F=FG=GF$. Define
  \begin{align*}
    \chi_t & := \chi_n\left(\frac{|x|^2}{t^2}\right), \\
    M_j & := |D|^{\alpha-2}D_j, \\
    Q & := x \cdot D |D|^{\alpha-2} = \sum_{j=1}^d x_jM_j.
  \end{align*}
  As $F=FG$,
  $$
  F\chi_tQ\chi_tF = F\chi_tGQ\chi_tF + F[G,\chi_t]Q\chi_tF.
  $$
  Taking adjoints and recalling that $F$, $G$, and $\chi_t$ are selfadjoint, we obtain
  \begin{align}
    \label{eq:principal-g-insertion}
    \frac{\alpha}{t^2} F\chi_t(Q+Q^*)\chi_tF
    = \frac{\alpha}{t^2} F\chi_t(GQ+\mathrm{h.c.})\chi_tF + \ce_{3,n}(t),
  \end{align}
  where
  \begin{align}
    \label{eq:R3-definition-corrected}
    \ce_{3,n}(t) := \frac{\alpha}{t^2} F[G,\chi_t]Q\chi_tF+\mathrm{h.c.}
  \end{align}

  \medskip
  \noindent
  {\bf Step 2: Bounding the leading term.}
  We now combine the first term on the right-hand side of \eqref{eq:principal-g-insertion} with the $-2|x|^2/t^3$ term to get
  \begin{align*}
    \mathcal{M}_n(t)
    := \alpha F \chi_t \frac{GQ + \text{h.c.}}{t^2} \chi_t F - F \chi_t^2 \frac{2|x|^2}{t^3} F.
  \end{align*}
  Since $|x|/t \sim R_n$ on $\supp\chi_n(B_t)$,
  \begin{align}
    \label{eq:revised-negative-drift}
    -\frac{2}{t} F\chi_t^2\frac{|x|^2}{t^2}F
    \lesssim -\frac{R_n^2}{t} F\chi_t^2F.
  \end{align}
  For the $GQ+\text{h.c.}$ term, let $u\in L^2(\R^d)$ and $v=\chi_tFu$. Since $Gx_j=x_jG+[G,x_j]$,
  \begin{align*}
    \frac1{t^2} |\langle v,Gx_jM_jv\rangle|
    \leq c\frac{R_n}{t} \|GM_j\|\,\|v\|^2 + \frac1{t^2} \|[G,x_j]M_j\|\,\|v\|^2.
  \end{align*}
  Set
  $$
  c_1 := \sum_{j=1}^d\|GM_j\|,
  \qquad
  c_2 := 2\sum_{j=1}^d\|[x_j,G]M_j\|.
  $$
  Both $c_1$ and $c_2$ are finite. Indeed, since $\alpha\leq2$,
  $$
  \langle H_\kappa\rangle^{-1}M_j
  = \left( \langle H_\kappa\rangle^{-1} \langle D\rangle^{\alpha/2} \right) \cdot \left( \langle D\rangle^{-\alpha/2}M_j \right) \in\mathcal B(L^2),
  $$
  which implies $GM_j\in\mathcal B(L^2)$, i.e., $c_1<\infty$. Furthermore, by the Helffer--Sj\"ostrand formula \eqref{eq:hs}, we get
  $$
  [x_j,G]
  = \frac{\alpha}{2\pi} \int_{\C} \bar\partial\widetilde g(z) (z-H_\kappa)^{-1} M_j (z-H_\kappa)^{-1} \,dz\wedge d\bar z.
  $$
  Hence, $[x_j,G]M_j\in\mathcal B(L^2)$, i.e., $c_2<\infty$. Thus, we conclude
  \begin{align}
    \label{eq:revised-velocity-upper}
    \frac{\alpha}{t^2} F\chi_t (GQ + \mathrm{h.c.}) \chi_tF 
    \lesssim \frac{R_n}{t} F\chi_t^2F + \frac{C}{t^2}.
  \end{align}
  Combining \eqref{eq:revised-negative-drift} and \eqref{eq:revised-velocity-upper}, we obtain
  $$
  \mathcal M_n(t)
  \leq -\frac{1}{t} \left[ C_1 R_n^2 - C_2 R_n \right] F\chi_t^2F + \frac{C_3}{t^2}.
  $$
  Since the expression in square brackets is to leading order $R_n^2$ for large $n$, there exist $N\in\N$ and $c>0$ such that
  $$
  \mathcal M_n(t) \leq -\frac{cR_n^2}{t} F\chi_t^2F + \frac{C_3}{t^2}, \quad n\geq N.
  $$
  Together with \eqref{eq:principal-g-insertion}, this proves the asserted inequality. The estimate on $\ce_{3,n}(t)$ is proved in Lemma~\ref{boundr3}.
\end{proof}

\begin{lemma}
  \label{boundr12}
  In the situation of Lemma~\ref{the7aux2}, assume $\alpha>2-d/2$. Then, uniformly in $n$ and $t\geq1$,
  \begin{align}
    \label{eq:boundr12-basic}
    \sum_{j=1}^2 \left| \left\langle \me{-itH_\kappa}\phi, \ce_{j,n}(t)\me{-itH_\kappa}\phi \right\rangle \right| \lesssim \frac{1}{R_n^2 t^{\alpha}} \|\langle x\rangle^2\phi\|^2.
  \end{align}
  If moreover $\alpha>3-d/2$, then, uniformly in $n$ and $t\geq1$,
  \begin{align}
    \label{eq:boundr12-improved}
    \sum_{j=1}^2 \left| \left\langle \me{-itH_\kappa}\phi, \ce_{j,n}(t)\me{-itH_\kappa}\phi \right\rangle \right| \lesssim \frac{1}{R_n^2} \|\langle x\rangle^2\phi\|^2
    \begin{cases}
      t^{-\min\{2,\alpha+\rho\}}, & V\neq0,\\
      t^{-2}, & V=0.
    \end{cases}
  \end{align}
\end{lemma}

\begin{proof}
  Recall $F:=f(H_\kappa)$ and $B_t=|x|^2/t^2$. We abbreviate $H = H_\kappa$ and $P = |D|^\alpha$,
  $$
  R_t(z) := (z-B_t)^{-1},
  \qquad
  \phi_t := \me{-itH}\phi,
  \qquad
  v_t := F\phi_t.
  $$
  Then, with this notation,
  \begin{align}
    \label{eq:boundr12-R1-dyadic}
    \ce_{1,n}(t)
    & = \frac1{2\pi} F \left(\int_{\C} \bar\partial\widetilde Y_{n,N}(z) R_t(z)^2 \ad_{B_t}^2(P) R_t(z) \,dz\wedge d\bar z\right) F
  \end{align}
  and
  \begin{align}
    \label{eq:boundr12-R2-dyadic}
    \ce_{2,n}(t) = i F\chi_n(B_t) [\chi_n(B_t),[P,B_t]] F.
  \end{align}
  
  \medskip
  \noindent
  {\bf Step 1: Expression for $\ad_{B_t}^2(P)$ and proof of \eqref{eq:boundr12-basic}.}
  Let
  $$
  M_{jk} := \partial_{\xi_j}\partial_{\xi_k}|\xi|^\alpha(D)
  = \alpha\delta_{jk}|D|^{\alpha-2} + \alpha(\alpha-2) |D|^{\alpha-4}D_jD_k.
  $$
  By iterating \eqref{eq:explicit-kinetic-commutator},
  \begin{align}
    \label{eq:exact-second-commutator-boundr12-dyadic}
    \ad_{B_t}^2(P)  
    = t^{-4}\ad_{|x|^2}^2(P)
    = -t^{-4}\sum_{j,k=1}^d \left(M_{jk}x_jx_k + 2x_jM_{jk}x_k + x_jx_kM_{jk}\right).
  \end{align}
  Since $\alpha>2-d/2$, we note that, by Hardy's inequality,
  \begin{align}
    \label{eq:Mjk-Hardy-dyadic}
    M_{jk}\langle x\rangle^{\alpha-2} \in\mathcal B(L^2).
  \end{align}
  We now commute the $x_j$ and $x_k$ next to $v_t$ to apply the propagation estimate in Lemma~\ref{lem9}.
  Consider first the middle ordering in \eqref{eq:exact-second-commutator-boundr12-dyadic}. Since $R_t(z)$ commutes with the multiplication operators $x_j$, we get, by \eqref{eq:Mjk-Hardy-dyadic},
  \begin{align}
    \label{eq:boundr12-basic-middle-dyadic}
    \frac1{t^4} \left| \left\langle v_t, R_t(z)^2x_jM_{jk}x_kR_t(z)v_t\right\rangle \right|
    \lesssim \frac{|\im z|^{-3}}{t^4} \|\langle x\rangle^{3-\alpha}v_t\| \|\langle x\rangle v_t\|.
  \end{align}
  By Lemma~\ref{lem9},
  $$
  \|\langle x\rangle^\mu v_t\| \lesssim_f t^\mu\|\langle x\rangle^\mu\phi\|, \quad 0\leq\mu\leq2,
  $$
  and $3-\alpha\leq2$, \eqref{eq:boundr12-basic-middle-dyadic} yields
  \begin{align}
    \label{eq:boundr12-basic-middle-final-dyadic}
    \frac1{t^4} \left| \left\langle v_t, R_t(z)^2x_jM_{jk}x_kR_t(z)v_t \right\rangle \right|
    \lesssim t^{-\alpha}|\im z|^{-3} \|\langle x\rangle^2\phi\|^2.
  \end{align}
  The two outer orderings involving $M_{jk}x_jx_k$ and $x_jx_kM_{jk}$ can be treated without commuting $M_{jk}$ through an $x$-factor. For example,
  \begin{align}
    \label{eq:boundr12-basic-outer-dyadic}
    \begin{split}
      \frac1{t^4} \left| \left\langle v_t, R_t(z)^2M_{jk}x_jx_kR_t(z)v_t \right\rangle\right|
      & \lesssim \frac{|\im z|^{-3}}{t^4} \|\langle x\rangle^2v_t\| \|\langle x\rangle^{2-\alpha}v_t\| \\
      & \lesssim t^{-\alpha} |\im z|^{-3} \|\langle x\rangle^2\phi\|^2.
    \end{split}
  \end{align}
  The term with $x_jx_kM_{jk}$ is estimated in the same way.
  Thus, by \eqref{eq:Yn-AA} with $k=0$ and $m=3$,
  \begin{align}
    \label{eq:boundr12-R1-basic-dyadic}
    |\langle\phi_t,\ce_{1,n}(t)\phi_t\rangle| \lesssim R_n^{-2}t^{-\alpha} \|\langle x\rangle^2\phi\|^2.
  \end{align}
  To estimate $\ce_{2,n}(t)$, set
  $$
  C_t := [P,B_t],
  \qquad
  D_t := [C_t,B_t] = \ad_{B_t}^2(P).
  $$
  Since $\chi_n\in C_c^\infty(\R)$, the Helffer--Sj\"ostrand formula yields
  \begin{align}
    \label{eq:R2-HS-boundr12-dyadic}
    [\chi_n(B_t),C_t] = -\frac1{2\pi i} \int_{\C} \bar\partial\widetilde\chi_{n,N}(z) R_t(z)D_tR_t(z) \,dz\wedge d\bar z.
  \end{align}
  Thus, the same three orderings as in $\mathcal E_{1,n}(t)$ occur, but with two $B_t$-resolvents instead of three. Repeating \eqref{eq:boundr12-basic-middle-dyadic}--\eqref{eq:boundr12-basic-outer-dyadic}, now with $|\im z|^{-2}$ instead of $|\im z|^{-3}$, and using~\eqref{eq:dyadic-chi-AA-integral} with $k=0$ and $m=2$, gives
  \begin{align}
    \label{eq:boundr12-R2-basic-dyadic}
    |\langle\phi_t,\ce_{2,n}(t)\phi_t\rangle| \lesssim R_n^{-2}t^{-\alpha} \|\langle x\rangle^2\phi\|^2.
  \end{align}
  This proves \eqref{eq:boundr12-basic}.

  \medskip
  \noindent
  {\bf Step 2: Preparations for the improved estimate \eqref{eq:boundr12-improved}.}
  To generate more decay in time when $\alpha\leq1$, we avoid commuting the $x_j$ and $x_k$ away from $M_{jk}$ as these yielded the terms $\|\langle x\rangle^\mu v_t\|$ yielding too high powers of $t$. Thus, to control ${\rm ad}_{B_t}^2(P)$, we use the energy localization more efficiently. To this end, we assume $\alpha>3-d/2$. In the following, set $\rho_0:=\min\{\rho,1\}$. Since $|V(x)| \lesssim \langle x\rangle^{-\rho_0}$ and $\min\{2,\alpha+\rho_0\} = \min\{2,\alpha+\rho\}$, we may, for the remainder of the proof, replace $\rho$ by $\rho_0$ and thus assume $0<\rho\leq1$.

  To exploit the energy localization, choose $g\in C_c^\infty((0,\infty);\R)$ such that $g=1$ on a neighborhood of $\supp f$. Set
  $$
  G := g(P),
  \qquad
  K:=g(H)-g(P).
  $$
  Since $g(H)F=F$,
  \begin{align}
    \label{eq:F-GK-decomposition-boundr12}
    v_t = Gv_t + Kv_t.
  \end{align}
  By the Helffer--Sj\"ostrand and resolvent formulae,
  \begin{align}
    \label{eq:K-definition-boundr12}
    K = -\frac1{2\pi i} \int_{\C} \bar\partial\widetilde g(\zeta) (\zeta-P)^{-1} \left( \kappa|x|^{-\alpha}+V \right) (\zeta-H)^{-1} \,d\zeta\wedge d\bar\zeta.
  \end{align}
  In Lemma~\ref{lem:xK} below, we prove
  \begin{align}
    \label{eq:xK-boundr12}
    \|x_k K v_t\| \lesssim \|\phi\| + \one_{\{V\neq0\}} t^{1-\rho} \|\langle x\rangle^{1-\rho}\phi\|.
  \end{align}
  Recalling $R_t(z)=(z-|x|^2/t^2)^{-1}$, we will also use the commutator estimate
  \begin{align}
    \label{eq:m-gP-commutator-boundr12}
    \|[x_k R_t(z),G]\| \lesssim_g |\im z|^{-1} + \langle z\rangle|\im z|^{-2}.
  \end{align}
  Indeed, $G=g(P)$ has a Schwartz convolution kernel because $g$ is supported away from the origin, while
  $$
  \left\|\nabla \left(x_k (z-|x|^2/t^2)^{-1}\right)\right\|_\infty
  \lesssim |\im z|^{-1} + \langle z\rangle|\im z|^{-2}.
  $$
  Thus, \eqref{eq:m-gP-commutator-boundr12} follows from Schur's test.

  \medskip
  \noindent
  {\bf Step 3: The $x_jM_{jk}x_k$ term.}
  We first consider again the middle ordering $x_jM_{jk}x_k$. Using \eqref{eq:F-GK-decomposition-boundr12},
  $$
  x_kR_t(z)v_t = Gx_kR_t(z)v_t + [x_kR_t(z),G]v_t + x_kR_t(z)Kv_t.
  $$
  Since $g$ vanishes in a neighborhood of the origin, $M_{jk}G\in\mathcal B(L^2)$. Hence, by Lemma~\ref{lem9}, the $Gx_kR_t(z)v_t$-part satisfies
  \begin{align}
    \label{eq:boundr12-free-main-dyadic}
    \frac1{t^4} \left| \left\langle v_t, R_t(z)^2x_jM_{jk}G x_kR_t(z)v_t \right\rangle \right|
    \lesssim t^{-2}|\im z|^{-3} \|\langle x\rangle\phi\|^2.
  \end{align}
  For the $[x_kR_t(z),G]v_t$-part, we move $M_{jk}$ to the opposite side and use~\eqref{eq:Mjk-Hardy-dyadic} and~\eqref{eq:m-gP-commutator-boundr12}. Thus, by $3-\alpha\leq2$ and Lemma~\ref{lem9},
  \begin{align}
    \label{eq:boundr12-free-comm-dyadic}
    \frac1{t^4} \left| \left\langle v_t, R_t(z)^2x_jM_{jk} [x_kR_t(z),G]v_t\right\rangle\right| 
    \lesssim t^{-1-\alpha} \left(|\im z|^{-3} + \langle z\rangle|\im z|^{-4}\right) \|\langle x\rangle^2\phi\|^2.
  \end{align}
  Finally, \eqref{eq:xK-boundr12} gives for the $x_kR_t(z)Kv_t$-part
  \begin{align}
    \label{eq:boundr12-K-middle-dyadic}
    \frac1{t^4} \left| \left\langle v_t, R_t(z)^2x_jM_{jk} x_kR_t(z)Kv_t \right\rangle \right|
      \lesssim |\im z|^{-3} \|\langle x\rangle^2\phi\|^2 \left(t^{-1-\alpha} + \one_{\{V\neq0\}}t^{-\alpha-\rho} \right).
  \end{align}
  Since $\alpha\geq1$, \eqref{eq:boundr12-free-main-dyadic}--\eqref{eq:boundr12-K-middle-dyadic} yield
  \begin{align}
    \label{eq:boundr12-middle-improved-dyadic}
    \begin{split}
      & \frac{1}{t^4}\left|\left\langle v_t,R_t(z)^2x_jM_{jk}x_kR_t(z)v_t\right\rangle\right| \\
      & \quad \lesssim \left(|\im z|^{-3} + \langle z\rangle|\im z|^{-4}\right)\|\langle x\rangle^2\phi\|^2
        \begin{cases}
          t^{-\min\{2,\alpha+\rho\}}, & V\neq0,\\
          t^{-2}, & V=0.
        \end{cases}
    \end{split}
  \end{align}
  
  \medskip
  \noindent
  {\bf Step 4: The $x_jx_kM_{jk}$ and $M_{jk}x_jx_k$ terms.}
  It remains to treat the terms involving $x_jx_kM_{jk}+M_{jk}x_jx_k$ in \eqref{eq:exact-second-commutator-boundr12-dyadic}. Commuting one $x$-factor past $M_{jk}$ reduces them to the middle ordering $x_jM_{jk}x_k$ considered in the previous step, together with terms involving
  $$
  N_{jk\ell}(D) := [M_{jk},x_\ell],
  $$
  which is a homogeneous Fourier multiplier of order $\alpha-3$ with smooth bounded angular part. Here we use the stronger hypothesis $\alpha>3-d/2$. By Hardy's inequality,
  \begin{align}
    \label{eq:N-Hardy-boundr12-dyadic}
    N_{jk\ell}(D)\langle x\rangle^{\alpha-3} \in\mathcal B(L^2).
  \end{align}
  Consider, e.g.,
  $$
  \frac{1}{t^4} R_t(z)^2N_{jk\ell}(D)x_kR_t(z)v_t
  = \frac{1}{t^4} R_t(z)^2N_{jk\ell}(D)x_kR_t(z)\left(Gv_t + Kv_t\right),
  $$
  where we used the decomposition of $v_t$ in \eqref{eq:F-GK-decomposition-boundr12}.
  The term in which $G$ is moved next to $N_{jk\ell}(D)$ is bounded by a constant times
  $$
  t^{-3}|\im z|^{-3} \|\langle x\rangle\phi\|\,\|\phi\|,
  $$
  by using Lemma~\ref{lem9} and the boundedness of $N_{jk\ell}(D)G$. The term containing the commutator $[x_kR_t(z),G]$ is, by \eqref{eq:N-Hardy-boundr12-dyadic}, \eqref{eq:m-gP-commutator-boundr12}, and $3-\alpha\leq2$, bounded by a constant times
  \begin{align*}
    t^{-1-\alpha} \left(|\im z|^{-3} + \langle z\rangle|\im z|^{-4}\right) \|\langle x\rangle^2\phi\|^2.
  \end{align*}
  Finally, by \eqref{eq:xK-boundr12},
  \begin{align}
    \label{eq:order-alpha-3-K-dyadic}
      \frac1{t^4} \left| \left\langle v_t, R_t(z)^2N_{jk\ell}(D)x_kR_t(z)Kv_t \right\rangle \right|
      \lesssim |\im z|^{-3} \|\langle x\rangle^2\phi\|^2 \left(t^{-1-\alpha} + \one_{\{V\neq0\}}t^{-\alpha-\rho} \right).
  \end{align}
  The terms with the remaining ordering are treated similarly.

  \medskip
  \noindent
  {\bf Step 5: Conclusion.}
  Plugging the bounds of Steps~3 and~4 in \eqref{eq:exact-second-commutator-boundr12-dyadic} shows
  \begin{align}
    \label{eq:all-orderings-improved-boundr12-dyadic}
    \begin{split}
      & \left| \left\langle v_t, R_t(z)^2 {\rm ad}_{B_t}^2(|D|^\alpha) R_t(z)v_t \right\rangle \right| \\
      & \quad \lesssim \left(|\im z|^{-3} + \langle z\rangle|\im z|^{-4}\right) \|\langle x\rangle^2\phi\|^2
        \begin{cases}
          t^{-\min\{2,\alpha+\rho\}}, & V\neq0,\\
          t^{-2}, & V=0.
        \end{cases}
    \end{split}
  \end{align}
  Integrating in $z$ and using \eqref{eq:Yn-AA} with $k\in\{0,1\}$ and $m\in\{3,4\}$, we obtain
  \begin{align}
    \label{eq:R1-final-improved-dyadic}
    |\langle\phi_t,\ce_{1,n}(t)\phi_t\rangle|
    \lesssim R_n^{-2} \|\langle x\rangle^2\phi\|^2
    \begin{cases}
      t^{-\min\{2,\alpha+\rho\}}, & V\neq0,\\
      t^{-2}, & V=0.
    \end{cases}
  \end{align}
  
  For $\ce_{2,n}(t)$ we use \eqref{eq:R2-HS-boundr12-dyadic}. The same second commutator $D_t=[[P,B_t],B_t]$ as in the previous steps occurs, but, compared to $\ce_{1,n}$, there is one fewer $B_t$-resolvent. Thus, the estimates from Steps~3 and~4 hold with $|\im z|^{-3} + \langle z\rangle|\im z|^{-4}$ replaced by $|\im z|^{-2} + \langle z\rangle|\im z|^{-3}$. Applying \eqref{eq:dyadic-chi-AA-integral} with $k\in\{0,1\}$ and $m\in\{2,3\}$ also yields
  \begin{align}
    \label{eq:R2-final-improved-dyadic}
    |\langle\phi_t,\ce_{2,n}(t)\phi_t\rangle|
    \lesssim R_n^{-2} \|\langle x\rangle^2\phi\|^2
    \begin{cases}
      t^{-\min\{2,\alpha+\rho\}}, & V\neq0, \\
      t^{-2}, & V=0.
    \end{cases}
  \end{align}
  Combining \eqref{eq:R1-final-improved-dyadic} and \eqref{eq:R2-final-improved-dyadic} proves \eqref{eq:boundr12-improved}.
\end{proof}

\begin{lemma}
  \label{boundr3}
  In the situation of Lemma~\ref{the7aux3}, one has
  \begin{align}
    \label{eq:boundr3-weaker-alpha}
    \left|
      \left\langle \me{-itH_\kappa}\phi,
      \ce_{3,n}(t)\me{-itH_\kappa}\phi\right\rangle
    \right|
    \lesssim t^{-2}\|\phi\|^2,
    \qquad \phi\in L^2(\R^d)
  \end{align}
  uniformly in $n\geq n_0$ and $t\geq1$.
\end{lemma}

\begin{proof}
  We abbreviate
  $$
  H:=H_\kappa,
  \qquad
  P:=|D|^\alpha,
  \qquad
  q_L(x):=q(|x|/L),
  \qquad L:=R_nt\geq1.
  $$
  Thus, $q_L=\chi_n(B_t)$. Recall $F=f(H)$, $G=g(H)$. Then, in this notation,
  $$
  \ce_{3,n}(t) = \frac{\alpha}{t^2} \sum_{j=1}^d \left(F[G,q_L] x_j\right) \left(|D|^{\alpha-2}D_j q_L F\right) + \mathrm{h.c.}.
  $$
  We first consider the first factor and prove
  \begin{align}
    \label{eq:xqG-first-order}
    \sup_{L\geq1}\|x_j[q_L,G]F\|<\infty,
    \quad j=1,...,d.
  \end{align}

  \medskip
  \noindent
  {\bf Step 1: Commutator bounds for the kinetic energy.}
  To prepare the proof of \eqref{eq:xqG-first-order}, we first show that, uniformly for $L\geq1$,
  \begin{align}
    \label{eq:qL-P-first-order}
    \|[q_L,P]u\| & \lesssim L^{-1}\|u\|_{H^{\alpha-1}},\\
    \label{eq:xqL-P-first-order}
    \|x_j[q_L,P]u\| & \lesssim \|u\|_{H^{\alpha-1}}.
  \end{align}
  To verify these estimates, note that if $a$ is a Schwartz function, then Plancherel's theorem, \eqref{eq:P-commutator} and Young's inequality yield \eqref{eq:qL-P-first-order}, as before. Applying this first with $a=q_L$, for which $\widehat q_L(\zeta)=L^d\widehat q(L\zeta)$ and $\left\||\zeta|^\alpha\widehat q_L\right\|_1 \lesssim L^{-\alpha}\leq L^{-1}$, yields \eqref{eq:qL-P-first-order}.

  Next let $a_{L,j}(x):=x_jq_L(x)$. Then, similarly as before, $\left\||\zeta|^\alpha\widehat a_{L,j}\right\|_1 \lesssim L^{1-\alpha}\leq1$. Thus, by \eqref{eq:first-moment-Fourier-bound},
  \begin{align}
    \label{eq:aLj-P-first-order}
    \|[a_{L,j},P]u\|\lesssim\|u\|_{H^{\alpha-1}}.
  \end{align}
  Next,
  \begin{align}
    \label{eq:xq-commutator-identity-first-order}
    x_j[q_L,P]=[x_jq_L,P]-[x_j,P]q_L
  \end{align}
  with $[x_j,P]=i\alpha|D|^{\alpha-2}D_j$ being a Fourier multiplier of order $\alpha-1$. Moreover, multiplication by $q_L$ is uniformly bounded on $H^{\alpha-1}$ because $0\leq\alpha-1\leq1$ and
  $$
  \sup_{L\geq1} \left(\|q_L\|_\infty+\|\nabla q_L\|_\infty\right)<\infty.
  $$
  Combining this observation with \eqref{eq:aLj-P-first-order} and \eqref{eq:xq-commutator-identity-first-order} proves \eqref{eq:xqL-P-first-order}.

  \medskip
  \noindent
  {\bf Step 2: Proof of \eqref{eq:xqG-first-order}.}
  Write $R(z):=(z-H)^{-1}$. Since $\kappa<\kappa_{\rm c}$ and $V$ is bounded, we have $\mathcal D((H+c)^{1/2})=H^{\alpha/2}(\R^d)$ for $c>0$ sufficiently large. Since $\alpha-1\leq\alpha/2$, we have
  \begin{align}
    \label{eq:localized-resolvent-Halpha-1}
    \|R(z)F\|_{L^2\to H^{\alpha-1}} \lesssim_{z,f}1
  \end{align}
  with at most polynomial growth in $|\im z|^{-1}$ and $\langle z\rangle$. Moreover, with
  $$
  K_j:=[x_j,H]=i\alpha|D|^{\alpha-2}D_j,
  $$
  the same argument yields
  \begin{align}
    \label{eq:Kj-resolvent-first-order}
    \|K_jR(z)\|+\|R(z)K_j\|\lesssim_z1,
  \end{align}
  again with at most polynomial growth in $|\im z|^{-1}$ and $\langle z\rangle$.

  By the Helffer--Sj\"ostrand formula,
  \begin{align}
    \label{eq:first-order-HS-qG}
    [q_L,G]F = \frac1{2\pi i}\int_\C \bar\partial\widetilde g(z) R(z)[q_L,P]R(z)F\,dz\wedge d\bar z.
  \end{align}
  Since $x_jR(z)=R(z)x_j+R(z)K_jR(z)$, the integrand after multiplication from the left by $x_j$ equals $\bar\partial\widetilde g(z)$ times
  \begin{align}
    \label{eq:x-first-order-HS-integrand}
    x_jR(z)[q_L,P]R(z)F = R(z)x_j[q_L,P]R(z)F + R(z)K_jR(z)[q_L,P]R(z)F.
  \end{align}
  The first term on the right is uniformly bounded in $L$ by \eqref{eq:xqL-P-first-order} and \eqref{eq:localized-resolvent-Halpha-1}. The second is uniformly bounded by \eqref{eq:qL-P-first-order}, \eqref{eq:localized-resolvent-Halpha-1}, and \eqref{eq:Kj-resolvent-first-order}. All bounds have only polynomial growth in $|\im z|^{-1}$ and $\langle z\rangle$. Thus, choosing the almost-analytic extension in \eqref{eq:first-order-HS-qG} of sufficiently high order and integrating proves \eqref{eq:xqG-first-order}.

  \medskip
  \noindent
  {\bf Step 3: Estimate of $\ce_{3,n}(t)$.}
  It remains to control $|D|^{\alpha-2}D_j q_LF$. Set
  $$
  M_j:=|D|^{\alpha-2}D_j,
  \qquad \beta:=\alpha-1\in[0,1].
  $$
  Then,
  \begin{align}
    \label{eq:MjF-bounded-weaker-alpha}
    M_jF = \left(M_j\langle D\rangle^{-\alpha/2}\right) \left(\langle D\rangle^{\alpha/2}F\right) \in\mathcal B(L^2)
  \end{align}
  since both factors are bounded since $\alpha-1\leq\alpha/2$.
  We also have, uniformly for $L\geq1$,
  \begin{align}
    \label{eq:Mj-qL-commutator-weaker-alpha}
    \|[M_j,q_L]\|\lesssim L^{-\beta}.
  \end{align}
  Indeed, for $0<\beta<1$, the convolution kernel of $M_j$ is bounded by $C|x|^{-d-\beta}$, while
  $$
  |q_L(x)-q_L(y)| \lesssim \min\left\{1,\frac{|x-y|}{L}\right\}.
  $$
  Schur's test yields \eqref{eq:Mj-qL-commutator-weaker-alpha}. For $\beta=0$, $M_j$ is a Riesz transform which is $L^2$-bounded. For $\beta=1$, $M_j=D_j$ and \eqref{eq:Mj-qL-commutator-weaker-alpha} follows from $[D_j,q_L]=-i\partial_jq_L$.

  Commuting $q_L$ past $M_j$ in the definition of $\ce_{3,n}(t)$, its first non-adjoint half is the sum over $j$ of
  \begin{align*}
    \ce_{3,1}^{(j)}(t) & := \frac\alpha{t^2}F[G,q_L]x_jq_LM_jF, \\
    \ce_{3,2}^{(j)}(t) & := \frac\alpha{t^2}F[G,q_L]x_j[M_j,q_L]F.
  \end{align*}
  Let $\phi_t:=\me{-itH}\phi$. By \eqref{eq:xqG-first-order} and \eqref{eq:MjF-bounded-weaker-alpha},
  \begin{align*}
    \left| \langle\phi_t,\ce_{3,1}^{(j)}(t)\phi_t\rangle \right|
    \lesssim t^{-2}\|\phi\|^2.
  \end{align*}
  Likewise, \eqref{eq:xqG-first-order} and \eqref{eq:Mj-qL-commutator-weaker-alpha} yield
  \begin{align*}
    \left| \langle\phi_t,\ce_{3,2}^{(j)}(t)\phi_t\rangle \right|
    \lesssim t^{-2}L^{1-\alpha}\|\phi\|^2
    \lesssim t^{-2}\|\phi\|^2
  \end{align*}
  since $L=R_nt\geq1$ and $\alpha\geq1$. Summing over $j\in\{1,...,d\}$ and adding the Hermitian conjugates concludes the proof \eqref{eq:boundr3-weaker-alpha}.
\end{proof}

We now give the proof of \eqref{eq:xK-boundr12}.

\begin{lemma}
  \label{lem:xK}
  Assume the hypotheses and notation of Step~2 of the proof of Lemma~\ref{boundr12}. Then, with $\rho\in(0,1]$,
  \begin{align}
    \label{eq:xK}
     \|x_kKv_t\| \lesssim_{f,V} \|\phi\| + \one_{\{V\ne0\}}t^{1-\rho} \|\langle x\rangle^{1-\rho}\phi\|, \quad t\geq1.
  \end{align}
\end{lemma}

\begin{proof}
  In the following, we write
  \begin{align*}
    P=|D|^\alpha,
    \qquad
    H=P-\frac{\kappa}{|x|^\alpha}-V.
  \end{align*}
  Recall also the notation
  $$
  K := g(H)-g(P),
  \qquad
  v_t:=f(H)\me{-itH}\phi,
  $$
  where $g=1$ near $\supp f$.
  Writing
  $$
  R_0(\zeta):=(\zeta-P)^{-1}
  \qquad
  R(\zeta):=(\zeta-H)^{-1},
  $$
  we may thus write
  \begin{align}
    \label{eq:K-resolvent}
    K=-\frac{1}{2\pi i}\int_{\C}\bar\partial\widetilde g(\zeta) R_0(\zeta)\left(\frac{\kappa}{|x|^\alpha}+V\right) R(\zeta)\,d\zeta\wedge d\bar\zeta.
  \end{align}
  We use
  $$
  x_kR_0(\zeta) = R_0(\zeta)x_k + i\alpha R_0(\zeta)^2|D|^{\alpha-2}D_k
  $$
  and consider in the following the contributions from $\kappa/|x|^\alpha$ and $V$ separately.
  For the term in which $x_k$ is multiplied with the Hardy potential, we write
  \begin{align}
    \label{eq:xHardy-factor}
    R_0(\zeta)x_k|x|^{-\alpha}R(\zeta)v_t
    & = \left(R_0(\zeta)|D|^{\alpha-1}\right) \left(|D|^{1-\alpha}x_k|x|^{-\alpha}\right) R(\zeta)v_t.
  \end{align}
  The first factor is bounded by polynomials of $|\im\zeta|^{-1}$ and $\langle\zeta\rangle$ since $\alpha-1\leq\alpha/2$. The last factor in~\eqref{eq:xHardy-factor} is bounded using the spectral localization with $f(H)$. Finally, consider the second factor. It is trivially bounded if $\alpha=1$. If $\alpha-1<d/2$, then the Riesz potential $|D|^{1-\alpha}$ has a positive kernel. Hence, for $u\in C_c^\infty$, we have
  $$
  \left| |D|^{1-\alpha}\frac{x_k}{|x|}|x|^{1-\alpha}u\right|
  \leq |D|^{1-\alpha}|x|^{1-\alpha}|u|.
  $$
  Thus, by Hardy's inequality,
  $$
  \left\||D|^{1-\alpha}(x_k/|x|)|x|^{1-\alpha}u\right\|_2 \leq C\|u\|_2,
  $$
  which shows boundedness of the second factor in \eqref{eq:xHardy-factor}.
  
  For the $R_0(\zeta)^2|D|^{\alpha-2}D_k$ term acting on the Hardy potential, we factorize
  \begin{align}
    \label{eq:comm-Hardy-factor}
    \begin{split}
      & R_0(\zeta)^2|D|^{\alpha-2}D_k|x|^{-\alpha}R(\zeta)v_t \\
      & \quad = \left(R_0(\zeta)^2|D|^{\alpha-2}D_k|D|^{\alpha/2}\right) \left(|D|^{-\alpha/2}|x|^{-\alpha/2}\right) \left(|x|^{-\alpha/2}R(\zeta)v_t\right).
    \end{split}
  \end{align}
  The first factor is a bounded Fourier multiplier, the second is bounded by Hardy's inequality, and the third is bounded since $R(\zeta):L^2\to H^{\alpha/2}$ is bounded as $\kappa<\kappa_{\rm c}$. Thus, the contributions from the Hardy potential are $\mathcal{O}(\|\phi\|)$.

  For the term in which $x_k$ is multiplied with \(V\), we use $|x_kV(x)|\lesssim\langle x\rangle^{1-\rho}$ and Lemmas~\ref{lem8} and~\ref{lem9} (recalling $\rho\in(0,1]$) to obtain
  \begin{align*}
    \|R_0(\zeta)x_kV R(\zeta)v_t\|
    \lesssim_\zeta \|\langle x\rangle^{1-\rho}R(\zeta)f(H)\me{-itH}\phi\| 
    \lesssim_\zeta t^{1-\rho} \|\langle x\rangle^{1-\rho}\phi\|
  \end{align*}
  with at most polynomial growth in $|\im\zeta|^{-1}$ and $\langle\zeta\rangle$.
  Finally,
  $$
  R_0(\zeta)^2|D|^{\alpha-2}D_kV R(\zeta)f(H) \in \mathcal B(L^2)
  $$
  since $R_0(\zeta)^2|D|^{\alpha-2}D_k\in\B(L^2)$, $V\in L^\infty$, and $R(\zeta)f(H)\in\B(L^2)$. This contribution is also $\mathcal{O}(\|\phi\|)$.

  Since all of the above bounds have at most polynomial growth in $|\im\zeta|^{-1}$ and $\langle\zeta\rangle$, choosing the almost-analytic extension in \eqref{eq:K-resolvent} of sufficiently high order, the $\zeta$-integral is finite. This concludes the proof of~\eqref{eq:xK}.
\end{proof}

We finally give the

\begin{proof}[Proof of Lemma~\ref{lem:late-start}]
  Let $H=H_\kappa$, $F=f(H)$, $U(t)=\me{-itH}$ and choose $n_0$ sufficiently large that Lemmas~\ref{the7aux2} and~\ref{the7aux3}, together with the estimates for $\ce_{j,n}(t)$, hold uniformly for $n\geq n_0$. Then, for fixed $n$,
  $$
  t\mapsto Q_n(t)=\langle FU(t)\phi,X_n(B_t)FU(t)\phi\rangle
  $$
  is continuously differentiable for every $\phi\in L^2$. Indeed, $HF$ is bounded, and $t\mapsto X_n(B_t)$ is continuously differentiable in operator norm on $(0,\infty)$ with $\partial_tX_n(B_t)=-2t^{-1} B_t\chi_n(B_t)^2$.
  Thus, $Q_n'(t)=\langle U(t)\phi,\mathbb D_{H}\Phi_n(t)U(t)\phi\rangle$ is the Heisenberg derivative considered in Lemmas~\ref{the7aux2} and~\ref{the7aux3} in the sense of quadratic forms.
  By Lemmas~\ref{boundr12} and~\ref{boundr3},
  $$
  \sum_{j=1}^3 |\langle U(t)\phi,\mathcal E_{j,n}(t)U(t)\phi\rangle|
  \lesssim R_n^{-2}t^{-\beta}\|\langle x\rangle^2\phi\|^2 + t^{-2}\|\phi\|^2.
  $$
  For $\phi\in\D(\langle x\rangle^2)$, these bounds and Lemmas~\ref{the7aux2} and~\ref{the7aux3} imply
  \begin{align}
    \label{eq:differential-ineq-dyadic}
    Q_n'(t) \leq \, -\frac{cR_n^2}{t}\|\chi_n(B_t)FU(t)\phi\|^2 + \frac C{t^2}\|\phi\|^2 + CR_n^{-2}t^{-\beta}\|\langle x\rangle^2\phi\|^2.
  \end{align}
  We now integrate \eqref{eq:differential-ineq-dyadic} over $t\in[S,T]$. Since $Q_n(T)\geq0$ and $\beta>1$, letting $T\to\infty$ yields
  \begin{align}
    \label{eq:shell-integral-pre}
    R_n^2\int_S^\infty \|\chi_n(B_t)FU(t)\phi\|^2\frac{dt}{t}
    \lesssim Q_n(S)+S^{-1}\|\phi\|^2 + R_n^{-2}S^{1-\beta}\|\langle x\rangle^2\phi\|^2.
  \end{align}
  The mutually disjoint shells $\{R_n\leq|x|/t<R_{n+1}\}$ cover $\{|x|/t\geq R_{n_0}\}$. As $\chi_n(B_t)=1$ on the corresponding shell, Tonelli's theorem and \eqref{eq:shell-integral-pre} yield
  \begin{align*}
    \int_S^\infty \|\one_{\{|x|/t\geq R_{n_0}\}}FU(t)\phi\|^2\frac{dt}{t}
    \lesssim \sum_{n\geq n_0}R_n^{-2} Q_n(S) + S^{-1}\|\phi\|^2 + S^{1-\beta}\|\langle x\rangle^2\phi\|^2.
  \end{align*}
  Here, we already used $\sum_{n\geq n_0}R_n^{-2}\lesssim1$ for the second and third summands. Since $X_n(s)\leq s$ for $s\geq0$,
  $$
  \sum_{n\geq n_0}R_n^{-2}Q_n(S)
  \lesssim S^{-2} \|xFU(S)\phi\|^2
  \lesssim S^{-2}\|\langle x\rangle\phi\|^2 + \|\phi\|^2,
  $$
  where the second inequality follows from Lemma~\ref{lem:refined-first-moment}. This proves \eqref{eq:late-start-bound} for $\Theta_0=R_{n_0}$, and hence for every $\Theta\geq\Theta_0$ by monotonicity of the exterior projections.
\end{proof}

%%%%%%%%%%%%%%%%%%%%%%%%%%%%%%%%%%%%%%%%
%%%%%%%%%%%%%%%%%%%%%%%%%%%%%%%%%%%%%%%%

\section{Proof of (\ref{eq:the7cor}) in Theorem~\ref{the7}}
\label{s:proofthe7cor}

The proof is inspired by that of~\cite[Theorem~2.34]{Isozaki2023}.
As before, we may and do without loss of generality assume that $f\in C_c^\infty(0,\infty)$ is real-valued. We write $H=H_\kappa$, $P=|D|^\alpha$, $F=f(H)$, and $U(t)=\me{-itH}$. We first prove the conclusion for $\phi\in\D(\langle x\rangle^\epsilon)$ for given fixed $\epsilon>0$.
Fix an admissible $\Theta_{0}>0$ for \eqref{eq:the7} in Theorem~\ref{the7} to hold, and set $\Theta_{1}=2\Theta_{0}$. Choose a nondecreasing function $\chi\in C^\infty(\R;[0,1])$ with
$$
\chi(s)=0\quad \text{for } s\leq\Theta_{0}^2,
\qquad
\chi(s)=1\quad \text{for } s\geq\Theta_{1}^2.
$$
For $t\geq1$, set
$$
B_t=\frac{|x|^2}{t^2},
\qquad
M_t=\chi(B_t),
\qquad
g(t)=\|M_t^{1/2}FU(t)\phi\|^2.
$$
We first prove \eqref{eq:the7cor} for $\phi\in\D(\langle x\rangle^\epsilon)$ and at the end extend it by a density argument to all $\phi\in L^2(\R^d)$.

\medskip
\noindent
{\bf Step 1: Reduction to an estimate for $\|g'\|_\infty$.}
Since
$$
0\leq g(t) \leq \left\|\one_{\{|x|/t\geq\Theta_{0}\}} FU(t)\phi\right\|^2,
$$
the time-integrated maximal velocity estimate \eqref{eq:the7} in Theorem~\ref{the7} yields
\begin{align}
  \label{eq:integrated-maximal-velocity-apriori}
  \int_1^\infty g(t)\,\frac{dt}{t}<\infty \quad\text{for } \phi\in\D(\langle x\rangle^\epsilon).
\end{align}
Since $\chi(s)=1$ for $s\geq\Theta_{1}^2$, every $\Theta\geq\Theta_{1}$ satisfies
$$
\left\|\one_{\{|x|/t\geq\Theta\}}FU(t)\phi\right\|^2 \leq g(t).
$$
Hence, it remains to show $g(t)\to0$ as $t\to\infty$. To this end, define $G(\tau)=g(\me{\tau})$ for $\tau\geq0$. Then $G\geq0$ and $G\in L^1([0,\infty))$. To estimate $G(\tau)$, we use for every $\delta>0$,
\begin{align}
  \label{eq:G-bound-integral-derivative}
  G(\tau) \leq \frac{1}{\delta}\int_\tau^{\tau+\delta}G(s)\,ds + \frac{\delta}{2}\|G'\|_\infty
\end{align}
as a consequence of the mean value theorem. We claim that $g$ is continuously differentiable and
\begin{equation}
  \label{eq:pointwise-time-derivative-bound}
  |g'(t)|\leq \frac{C_{f,\chi}}{t}\|\phi\|^2, \quad t\geq1.
\end{equation}
Once we have shown this, we can in \eqref{eq:G-bound-integral-derivative} first let $\tau\to\infty$, using \eqref{eq:integrated-maximal-velocity-apriori}, and then $\delta\searrow0$, to get $\lim_{\tau\to\infty}G(\tau)=0$, and hence $\lim_{t\to\infty}g(t)=0$.

\medskip
\noindent
{\bf Step 2: Proof of \eqref{eq:pointwise-time-derivative-bound}.}
In the following, we justify the Heisenberg derivative
\begin{align}
  \label{eq:g-pointwise-derivative}
  \frac{d}{dt} \left(U(t)^* F M_t F U(t)\right)
  = U(t)^* \left(F \, \frac{dM_t}{dt}\, F + i F[P,M_t]F \right)U(t),
\end{align}
and estimate its right-hand side.
Since $\kappa<\kappa_{\rm c}$ and $V$ is bounded, there is $c>0$ such that $H+c\geq1$ and $\D((H+c)^{1/2})=H^{\alpha/2}(\R^d)$. Since $f$ is compactly supported and $\alpha-1\leq\alpha/2$, it follows that $HF$ is bounded on $L^2$ and
\begin{align}
  \label{eq:pointwise-spectral-regularity}
  \|Fu\|_{H^{\alpha-1}}\leq C_f\|u\|, \quad u\in L^2(\R^d).
\end{align}
We first estimate the $[P,M_t]$ term. To this end, we write $M_t=1+a_t$, where $a_t(x)=a(x/t)$ with
$$
a(x)=\chi(|x|^2)-1
$$
being in $C_c^\infty(\R^d)$. For $1\leq\alpha\leq2$, the bound \eqref{eq:P-commutator} and the Fourier representation of the commutator yield, by Young's inequality,
\begin{align*}
  \|[P,a_t]u\|
  \lesssim \left\| |\zeta|\widehat a_t(\zeta)\right\|_{L^1} \|u\|_{H^{\alpha-1}} + \left\||\zeta|^\alpha\widehat a_t(\zeta)\right\|_{L^1} \|u\| 
  \lesssim t^{-1}\|u\|_{H^{\alpha-1}} + t^{-\alpha}\|u\|.
\end{align*}
Here we used $\widehat a_t(\zeta)=t^d\widehat a(t\zeta)$. Thus, \eqref{eq:pointwise-spectral-regularity} implies
\begin{align}
  \label{eq:pointwise-commutator-bound}
  \|[P,a_t]F\| \leq C_{f,\chi}t^{-1},\quad t\geq1.
\end{align}
Therefore, $F[P,a_t]F$ extends to a bounded operator with
$$
\|F[P,a_t]F\| \leq C_{f,\chi}t^{-1}, \quad t \geq1.
$$
Moreover, $t\mapsto M_t$ is continuously differentiable in operator norm on $(0,\infty)$, with
$$
\frac{d}{dt} M_t=-\frac{2}{t}B_t \chi'(B_t),
\qquad
\left\|\frac{d}{dt} M_t\right\| \leq \frac{2}{t}\sup_{s\in\R}|s\chi'(s)|.
$$
Since $f$ has compact support, $t\mapsto FU(t)$ is also continuously differentiable in operator norm. This shows~\eqref{eq:g-pointwise-derivative} and the previous estimates yield \eqref{eq:pointwise-time-derivative-bound}.

\medskip
\noindent
{\bf Step 3: Extension from $\D(\langle x\rangle^\epsilon)$ to $L^2(\R^d)$.}
Finally, let $\phi\in L^2(\R^d)$ be arbitrary and choose $\phi_n\in\D(\langle x\rangle^\epsilon)$ with $\phi_n\to\phi$ in $L^2$. The uniform bound
$$
\left\|\one_{\{|x|/t\geq\Theta\}}FU(t)\right\| \leq\|f\|_\infty
$$
and the conclusion already proved for $\phi_n$ imply
$$
\limsup_{t\to\infty} \left\|\one_{\{|x|/t\geq\Theta\}}FU(t)\phi\right\| \leq\|f\|_\infty\|\phi-\phi_n\|.
$$
Letting $n\to\infty$ concludes the proof of \eqref{eq:the7cor}.
\qed

%%%%%%%%%%%%%%%%%%%%%%%%%%%%%%%%%%%%%%%%
%%%%%%%%%%%%%%%%%%%%%%%%%%%%%%%%%%%%%%%%

\section{Proof of Theorem~\ref{asympcomplhardyplusperturbationnonexistence}}
\label{s:proofnonexistence}

We prove Theorem~\ref{asympcomplhardyplusperturbationnonexistence} via contradiction. That is, we assume that the wave operator $W(H_\kappa,L_\kappa)$ exists. To arrive at a contradiction, we first establish the existence of suitable test functions. The following lemma is stronger than necessary, but shows the usefulness of Lemma~\ref{lem8}.

\begin{lemma}
  \label{lem:wave-test-vector}
  Let $\alpha\in[1,2]\cap(0,d)$ and $\kappa<\kappa_{\rm c}$. If $d=2$, assume $\alpha>1$.
  Let $I=[a,b]$ with some $0<a<b$. Then, there exist $0\ne\phi\in L^2(\R^d)$ and $f\in C_c^\infty(I)$ such that
  $$
  \phi = f(L_\kappa)\phi,\qquad
  \phi\in\D(\langle x\rangle^2)\cap\D(\langle A\rangle^2).
  $$
\end{lemma}

\begin{proof}
  Choose $f_0,f\in C_c^\infty(I)$ with $f=1$ near $\supp f_0$ and $f_0\not\equiv0$. Since $L_\kappa$ has purely absolutely continuous spectrum $[0,\infty)$ (as a consequence of Mourre theory and the virial theorem; see, e.g., \cite{IshidaMerz2026}), we have $f_0(L_\kappa)\ne0$. Since $C_c^\infty(\R^d\setminus\{0\})$ is dense in $L^2$, choose $\psi\in C_c^\infty(\R^d\setminus\{0\})$ with $\phi:=f_0(L_\kappa)\psi\ne0$. Then $f(L_\kappa)\phi=\phi$. By the Helffer--Sj\"ostrand formula and Lemma~\ref{lem8} with weights at most two, we have $\phi\in\D(\langle x\rangle^2)$. By $i[L_\kappa,A]=\alpha L_\kappa$, we have $(L_\kappa-z)^{-1}\in C^\infty(A)$ with $z\notin\R$. Hence, expanding $A^2\phi=A^2 f_0(L_\kappa)\psi$ in commutators, using the boundedness of
  $$
  [A,f_0(L_\kappa)]=i\alpha Lf_0'(L_\kappa),
  \qquad
  [A,[A,f_0(L_\kappa)]]=-\alpha^2\left(L_\kappa f_0'(L_\kappa) + L_\kappa^2 f_0''(L_\kappa)\right),
  $$
  and $\psi\in\D(A^2)$ shows $\phi\in\D(A^2)$.
\end{proof}

\begin{proof}[Proof of Theorem~\ref{asympcomplhardyplusperturbationnonexistence}]
  Write $H_\kappa=L_\kappa-V^L$ and recall that $P_{\rm ac}(L_\kappa)$ is the identity.
  Suppose, for contradiction, that
  $$
  W = \slim_{t\to\infty}\me{itH_\kappa}\me{-itL_\kappa}
  $$
  exists. Choose the nonzero $\phi$ and $f$ from Lemma~\ref{lem:wave-test-vector}, and set $u_t=\me{-itL_\kappa}\phi$. Choose $\lambda$ so large that
  $$
  R_\lambda = (\lambda+H_\kappa)^{-1},
  \qquad
  0 < R_\lambda\leq\mu^{-1},
  \qquad
  \mu>K_2^2/K_1.
  $$
  For $0<t_1<t_2$, Duhamel's formula and the intertwining property yield
  \begin{align}
    \label{eq:wave-Duhamel}
    \langle W\phi, (\me{it_2H_\kappa}\me{-it_2L_\kappa}-\me{it_1H_\kappa}\me{-it_1L_\kappa})\phi\rangle = I_1(t_1,t_2)+I_2(t_1,t_2),
  \end{align}
  where
  \begin{align}
    \label{eq:wave-I1}
    & I_1 = -i\int_{t_1}^{t_2} \langle u_t, (\lambda+L_\kappa)R_\lambda V^Lu_t\rangle\,dt, \\
    \label{eq:wave-I2}
    & I_2 = -i\int_{t_1}^{t_2} \langle \me{-itH_\kappa}(W-\me{itH_\kappa}\me{-itL_\kappa})(\lambda+L_\kappa)\phi, R_\lambda V^L u_t\rangle\,dt.
  \end{align}
  We now estimate $|I_1|$ from below and $|I_2|$ from above, and begin with $I_1$.
  By the resolvent formula,
  \begin{align}
    \label{eq:wave-positive-resolvent}
    \langle u_t,(\lambda+L_\kappa)R_\lambda V^Lu_t\rangle
    = \langle u_t, V^Lu_t\rangle + \langle V^Lu_t, R_\lambda V^Lu_t\rangle.
  \end{align}
  If $V^L\geq K_1\langle x\rangle^{-1}$, both terms have the required positive sign. If $V^L\leq-K_1\langle x\rangle^{-1}$, the absolute value of the second term is bounded from above by
  $$
  \mu^{-1}K_2^2\langle u_t, \langle x\rangle^{-2}u_t\rangle
  \leq \mu^{-1}K_2^2\langle u_t, \langle x\rangle^{-1}u_t\rangle.
  $$
  Since $\mu>K_2^2/K_1$, the integrand in \eqref{eq:wave-I1} has a fixed sign and
  \begin{align}
    \label{eq:I1-first-lower}
    |I_1| \geq \left(K_1-\frac{K_2^2}{\mu}\right) \int_{t_1}^{t_2} \langle u_t, \langle x\rangle^{-1}u_t\rangle\,dt.
  \end{align}
  Choose $\Theta\geq1$ at least as large as the threshold in the maximal velocity estimate \eqref{eq:the7} in Theorem~\ref{the7} for $L_\kappa$ and $f$, and set $P_t=\one\left(|x|<\Theta t\right)$. Since $\langle x\rangle^{-1}\geq(\sqrt2\Theta t)^{-1}$ on $\ran P_t$ and $\|u_t\|=\|\phi\|$,
  $$
  \langle u_t, \langle x\rangle^{-1}u_t\rangle
  \geq \frac{1}{\sqrt2\Theta t} \|P_tu_t\|^2
  = \frac{1}{\sqrt2\Theta t}\left(\|\phi\|^2 - \|(1-P_t)u_t\|^2\right).
  $$
  By Theorem~\ref{the7},
  $$
  \int_1^\infty \|(1-P_t)u_t\|^2\,\frac{dt}{t}
  \leq C \|\langle x\rangle^2\phi\|^2.
  $$
  Thus, with $c_0=\frac{1}{\sqrt2\Theta} \left(K_1-\frac{K_2^2}{\mu}\right)>0$,
  \begin{align}
    \label{eq:I1-final}
    |I_1| \geq c_0\|\phi\|^2 \log(t_2/t_1) - C_0\|\langle x\rangle^2\phi\|^2.
  \end{align}

  \smallskip
  We now estimate $I_2$. For every $\epsilon>0$, the assumed existence of $W$ provides $t_1(\epsilon)$ such that, for $t\geq t_1$,
  $$
  \| (W-\me{itH_\kappa}\me{-itL_\kappa})(\lambda+L_\kappa)\phi\|
  \leq \epsilon \|(\lambda+L_\kappa)\phi\|.
  $$
  Let $\chi\in C^\infty(\R;[0,1])$ be the cutoff in Lemma~\ref{lem:commutators-spectral-cutoff}, i.e., $\chi(s)=0$ for $s\leq\theta/2$ and $\chi(s)=1$ for $s\geq\theta$. Splitting $1=\left(1-\chi(A/t)\right)+\chi(A/t)$, the minimal velocity estimate (Theorem~\ref{the6} with $V=0$) with $\nu=2$, $\|(1-\chi(A/t))u\|\leq\|\one\left(\frac{A}{t}\leq\theta\right)u\|$, and $\phi\in\D(\langle A\rangle^2)$ give
  \begin{align}
    \label{eq:I2-low-A}
    \int_{t_1}^{t_2} \|R_\lambda V^L(1-\chi(A/t))u_t\|\,dt
    \leq C \|\langle A\rangle^2\phi\|.
  \end{align}
  For the $\chi(A/t)$ term, Lemma~\ref{lem:commutators-spectral-cutoff} gives
  \begin{align*}
    \|R_\lambda V^L \chi(A/t)u_t\|
    & \leq \|\chi(A/t)R_\lambda V^L u_t\| + C t^{-1} \|\phi\| \\
    & \leq  \frac{2}{\theta t} \|AR_\lambda V^L f(L_\kappa)u_t\| + Ct^{-1} \|\phi\|
      \leq Ct^{-1} \|\phi\|.
  \end{align*}
  In the last line we used $u_t=f(L_\kappa)u_t$ and \eqref{eq:ARVf}. Therefore,
  \begin{align}
    \label{eq:I2-final}
    |I_2| \leq \epsilon C\|(\lambda+L_\kappa)\phi\| \left( \|\langle A\rangle^2\phi\| + \|\phi\| \log(t_2/t_1) \right).
  \end{align}
  The left side of \eqref{eq:wave-Duhamel} is bounded uniformly in $t_1,t_2$. Choose $\epsilon>0$ so small that
  $$
  c_0 \|\phi\|^2 - \epsilon C\|(\lambda+L_\kappa)\phi\| \|\phi\|>0.
  $$
  Then \eqref{eq:I1-final} and \eqref{eq:I2-final} force the right side of \eqref{eq:wave-Duhamel} to tend to infinity as \(t_2\to\infty\), a contradiction. Hence, $W(H_\kappa,L_\kappa)$ does not exist.
\end{proof}

%%%%%%%%%%%%%%%%%%%%%%%%%%%%%%%%%%%%%%%% 
%%%%%%%%%%%%%%%%%%%%%%%%%%%%%%%%%%%%%%%%

\appendix
\section{Necessity of a logarithmic moment in (\ref{eq:the7})}
\label{s:sharp}

We show that \eqref{eq:the7} cannot hold with a weight $\log(2+\langle x\rangle)^\gamma$, $\gamma\in[0,1/2)$. To see this, let $\kappa=V=0$, $0\neq f\in C_c^\infty(0,\infty)$, and choose $\psi$ with $\hat\psi \in C_c^\infty(\R^d)$ supported away from the origin such that $v:=f(L_0)\psi\neq0$. Then, with $m:=\|v\|>0$, $A:=\|xv\|$, and $B:=\alpha\||D|^{\alpha-2}Dv\|$,
$$
\|x\me{-itL_0}v\| \leq A+Bt, \quad t\geq0.
$$
For $R>0$, let $\psi_R(x)=\psi(x-Re_1)$ with $e_1$ denoting the first canonical unit vector in $\R^d$. By translation invariance, $f(L_0)\me{-itL_0}\psi_R(x)=(\me{-itL_0}v)(x-Re_1)$. If $|x|\leq\Theta t$ and $\Theta t<R$, then $|x-Re_1|\geq R-\Theta t$. Thus,
$$
\|\one_{\{|x|\leq\Theta t\}}f(L_0)\me{-itL_0}\psi_R\|
\leq \frac{A+Bt}{R-\Theta t}.
$$
Choose $c>0$ such that $\Theta c\leq1/2$ and $2Bc\leq m/4$. For $R$ large enough that $2A/R\leq m/4$, the last expression is at most $m/2$ for $1\leq t\leq cR$. Thus, by orthogonality,
$$
\int_1^{cR} \|\one_{\{|x|/t\geq\Theta\}}f(L_0)\me{-itL_0}\psi_R\|^2\,\frac{dt}{t}
\geq \frac{3m^2}{4}\log(cR).
$$
Since $\|\log(2+\langle x\rangle)^\gamma\psi_R\|^2 \lesssim (\log R)^{2\gamma}$, this shows that \eqref{eq:the7} cannot hold for $\gamma\in[0,1/2)$.

In fact, there exists $\phi\in L^2(\R^d)$ for which the left-hand side of \eqref{eq:the7} is infinite: consider
$$
(T_N\phi)(t):=\one_{[1,N]}(t)\one_{\{|x|/t\geq\Theta\}}f(L_0)\me{-itL_0}\phi
$$
as bounded operators from $L^2(\R^d)$ into $L^2([1,\infty),dt/t;L^2(\R^d))$. If the full $t$-integral was finite for every $\phi$, then $\sup_N\|T_N\phi\|<\infty$ for all $\phi$. Thus, by the uniform boundedness principle, we would obtain $\sup_N\|T_N\|<\infty$. Taking $R=N/c$ in the preceding lower bound contradicts this conclusion as $N\to\infty$.

%%%%%%%%%%%%%%%%%%%%%%%%%%%%%%%%%%%%%%%%
%%%%%%%%%%%%%%%%%%%%%%%%%%%%%%%%%%%%%%%%

\section*{Acknowledgments}

A.~I.~was supported by JSPS KAKENHI Grant Numbers JP21K03279, JP26K06834.
K.~M.~thanks Katsushika Division, Institute of Arts and Sciences of Tokyo University of Science for their hospitality during a stay in April and May 2026, when part of this work was carried out.
We are grateful to Gian Michele Graf for fruitful discussions.
The authors acknowledge the use of AI tools for the final proofreading of the manuscript. All mathematical arguments and proofs in the final manuscript were checked and written by the authors.

% \bibliographystyle{alpha}
% \bibliography{coulomb}

\def\cprime{$'$}

\end{document}